\documentclass[leqno]{amsart}

\numberwithin{equation}{section} 
\usepackage{amsmath,amsfonts,amssymb,amsthm,latexsym}
\usepackage{enumerate}
\usepackage{cancel}
\usepackage{verbatim}
\usepackage{cases}
\usepackage[square,comma,numbers,sort&compress]{natbib}

\usepackage[colorlinks=true]{hyperref}
\hypersetup{urlcolor=blue, citecolor=red}

\usepackage{multicol, color}

\newcounter{mnote}
\usepackage{mathrsfs}

\theoremstyle{plain}
\newtheorem{theorem}{Theorem}[section]

\newtheorem{lemma}[theorem]{Lemma}

\theoremstyle{definition}
\newtheorem{definition}[theorem]{Definition}

\theoremstyle{remark}
\newtheorem{remark}[theorem]{Remark}

\newcommand{\vect}[1]{\mathbf{#1}}

\newcommand{\bk}{\vect{k}}

\newcommand{\bu}{\vect{u}}
\newcommand{\bv}{\vect{v}}
\newcommand{\bw}{\vect{w}}
\newcommand{\bm}{\vect{m}}
\newcommand{\bx}{\vect{x}}
\newcommand{\by}{\vect{y}}

\newcommand{\be}{\vect{e}}

\newcommand{\field}[1]{\mathbb{#1}}
\newcommand{\nN}{\field{N}}
\newcommand{\nZ}{\field{Z}}

\newcommand{\nR}{\field{R}}

\newcommand{\cD}{\mathcal D}

\newcommand{\bun}{\vect{u}^{(n)}}
\newcommand{\thetan}{\theta^{(n)}}
\newcommand{\omegan}{\omega^{(n)}}

\newcommand{\nT}{\mathbb T}
\newcommand{\vphi}{\varphi}
\newcommand{\maps}{\rightarrow}

\newcommand{\sand}{\quad\text{and}\quad}

\newcommand{\tac}{\textasteriskcentered}

\newcommand{\pd}[2]{\frac{d #1}{d #2}}

\newcommand{\abs}[1]{\left\lvert#1\right\rvert}
\newcommand{\norm}[1]{\left\lVert#1\right\rVert}
\newcommand{\set}[1]{\left\{#1\right\}}
\newcommand{\ip}[2]{\left<#1,#2\right>}
\newcommand{\pnt}[1]{\left(#1\right)}
\newcommand{\pair}[2]{\left(#1,#2\right)}

\newcommand{\adv}[2]{(#1 #2)}

\newcommand{\diff}[1]{\widetilde{#1}}
\newcommand{\bud}{\diff{\bu}}
\newcommand{\xid}{\diff{\xi}}
\newcommand{\thetad}{\diff{\theta}}

\begin{document}
\title[Boussinesq Equations]{Global Well-posedness for The 2D Boussinesq System Without Heat Diffusion and With Either Anisotropic Viscosity or Inviscid Voigt-$\alpha$ Regularization}

\date{October 24, 2010.}

\author{Adam Larios}
\address[Adam Larios]{Department of Mathematics\\
                University of California, Irvine\\
        Irvine CA 92697-3875, USA}
\email[Adam Larios]{alarios@math.uci.edu}
\author{Evelyn Lunasin}
\address[Evelyn Lunasin]{Department of Mathematics\\
    University of Michigan\\
    Ann Arbor, MI 48104 USA}
\email[Evelyn Lunasin]{lunasin@umich.edu}
\author{Edriss S. Titi}
\address[Edriss S. Titi]{Department of Mathematics, and Department of Mechanical and Aero-space Engineering\\
University of California, Irvine\\
Irvine CA 92697-3875, USA.\\
Also The Department of Computer Science and Applied Mathematics\\
The Weizmann Institute of Science, Rehovot 76100, Israel}
\email[Edriss S. Titi]{etiti@math.uci.edu and edriss.titi@weizmann.ac.il}


\begin{abstract}
We establish global existence and uniqueness theorems for the two-dimensional non-diffusive Boussinesq system with viscosity only in the horizontal direction, which arises in Ocean dynamics.  This work improves the global well-posedness results established recently by R. Danchin and M. Paicu for the Boussinesq system with anisotropic viscosity and zero diffusion.  Although we follow some of their ideas, in proving the uniqueness result, we have used an alternative approach by writing the transported temperature (density) as $\theta = \Delta\xi$ and adapting the techniques of V. Yudovich for the 2D incompressible Euler equations.   This new idea allows us to establish uniqueness results with fewer assumptions on the initial data for the transported quantity $\theta$.  Furthermore, this new technique allows us to establish uniqueness results without having to resort to the paraproduct calculus of J. Bony.

We also propose an inviscid $\alpha$-regularization
for the two-dimensional inviscid, non-diffusive
Boussinesq system of equations, which we call the
Boussinesq-Voigt equations.  Global regularity of
this system is established.  Moreover, we establish
the convergence of solutions of the
Boussinesq-Voigt model to the corresponding
solutions of the two-dimensional Boussinesq system
of  equations for inviscid flow without heat
(density) diffusion on the interval of existence of
the latter.  Furthermore, we derive a criterion for
finite-time blow-up of the solutions to the
inviscid, non-diffusive 2D Boussinesq system based
on this inviscid Voigt regularization. Finally, we
propose a Voigt-$\alpha$ regularization for the
inviscid 3D Boussinesq equations with diffusion,
and prove its global well-posedness.  It is worth
mentioning that our results are also valid in the
presence of the $\beta$-plane approximation of the
Coriolis force.
  \end{abstract}

 \maketitle
 \thispagestyle{empty}

\section{Introduction}\label{sec:Int}
The $d$-dimensional Boussinesq system of ocean and atmosphere dynamics (without rotation) in a domain $\Omega\subset\nR^d$ over the time interval $[0,T]$ is given by
\begin{subequations}\label{bouss}
\begin{alignat}{2}
\label{bouss_mo}
\partial_t\bu + \sum_{j=1}^d\partial_j(u^j \bu) &=-\nabla p + \theta \be_d + \nu\triangle\bu,
\qquad&& \text{in }\Omega\times[0,T],\\
\label{bouss_div}
\nabla \cdot \bu &=0,
\qquad&& \text{in }\Omega\times[0,T],\\
\label{bouss_den}
\partial_t\theta + \nabla\cdot(\bu \theta) &=\kappa\triangle\theta,
\qquad&& \text{in }\Omega\times[0,T],\\
\label{bouss_IC}
\bu(\bx,0)&=\bu_0(\bx),\quad\theta(\bx,0)=\theta_0(\bx),
\qquad&& \text{in }\Omega,
\end{alignat}
\end{subequations}
with appropriate boundary conditions (discussed below).  Here $\nu\geq0$ is the fluid viscosity, $\kappa\geq0$ is the diffusion coefficient.  The spatial variable is denoted $\bx=(x^1,\ldots,x^d)\in\Omega$, and the unknowns are the fluid velocity field $\bu\equiv\bu(\bx,t)\equiv(u^1(\bx,t),\ldots,u^d(\bx,t))$, the fluid pressure $p(\bx,t)$, and the function $\theta\equiv\theta(\bx,t)$, which may be interpreted physically as a thermal variable (e.g., when $\kappa>0$), or a density variable (e.g., when $\kappa=0$).  We write $\be_d=(0,\ldots,0,1)$ for the $d^{\text{th}}$ standard basis vector in $\nR^d$.  We use the notation $P^0_{\nu,\kappa}$, for the Boussinesq system with viscosity $\nu>0$ and with diffusion $\kappa>0$.    We attach a subscript $x$ to the viscosity $\nu$ when we mean that the viscosity occurs in the horizontal direction only, i.e. in the case of anisotropic viscosity (see equation \eqref{bouss_aniso} below).  The superscript of zero is reserved for a parameter $\alpha$, introduced below.

In two dimensions, the global regularity in time of the problem $P^0_{\nu,\kappa}$  is well-known (see, e.g., \cite{Cannon_DiBenedetto_1980, Temam_1997_IDDS}), and follows essentially from the classical methods for Navier-Stokes equations.  However, in the case $\nu=0, \kappa=0$, ($P^0_{0,0}$), global existence and uniqueness still remains an open problem  (see, e.g., \cite{Chae_2005,Chae_Nam_1997} for studies in this direction).  The local existence and uniqueness of classical solutions to $P_{0,0}^0$ was established in \cite{Chae_Nam_1997}, assuming the initial data $(\bu_0,\theta_0)\in  H^3\times H^3$.   In particular, an analogous Beale-Kato-Majda criterion for blow-up of smooth solutions is established in \cite{Chae_Nam_1997} for the inviscid, non-diffusive Boussinesq system; namely, that the smooth solution exists on $[0,T]$ if and only if $\int_0^T\|\nabla\theta(t)\|_{L^\infty }\;dt < \infty$.

One of our main results in this study, discussed in Section \ref{sec:aniso}, involves the global existence and uniqueness theorems for the two-dimensional non-diffusive Boussinesq system with viscosity only in the horizontal direction, denoted as $P^0_{\nu_x,0}$ (see equations \eqref{bouss_aniso} below).  These equations are sometimes called the non-diffusive Boussinesq equations with anisotropic viscosity.  In order to set the main ideas of our proof, in Section \ref{s:P_nu_zero_zero} we first establish the global existence of a certain class of weak solutions and the global existence and uniqueness to the two-dimensional viscous and non-diffusive  Boussinesq system of equations (denoted as $P^0_{\nu,0}$) with Yudovich-type initial data.  The other main result we have in this study is presented in Section \ref{s:P_alpha_zero_zero}.  We propose an inviscid $\alpha$-regularization for the two-dimensional inviscid, non-diffusive Boussinesq system of equations (denoted as $P^{\alpha}_{0,0}$), which we call the Boussinesq-Voigt equations, and also establish its global regularity.  We include in this section a study of the behavior of solutions to $P^{\alpha}_{0,0}$ as the parameter $\alpha\rightarrow 0$, which leads to a new criterion for the finite-time blow-up of solutions to the 2D, or 3D, inviscid, non-diffusive Boussinesq equations.   We also give a short discussion of a Voigt-regularization for the three-dimensional Boussinesq equations in the case $P^{\alpha}_{0,\kappa}$.   

The two-dimensional viscous, non-diffusive Boussinesq system, ($P^0_{\nu, 0}$) is given by:
\begin{subequations}\label{bouss_nd}
\begin{alignat}{2}
\label{bouss_nd_mo}
\partial_t\bu + \sum_{j=1}^2\partial_j(u^j \bu) &=\nu\Delta \bu-\nabla p + \theta \be_2,
\qquad&& \text{in }\nT^2\times[0,T],\\
\label{bouss_nd_div}
\nabla \cdot \bu &=0,
\qquad&& \text{in }\nT^2\times[0,T],\\
\label{bouss_nd_den}
\partial_t\theta + \nabla\cdot(\bu \theta) &=0,
\qquad&& \text{in }\nT^2\times[0,T],\\
\label{bouss_nd_IC}
\bu(\bx,0)&=\bu_0(\bx),\quad\theta(\bx,0)=\theta_0(\bx),
\qquad&& \text{in }\nT^2.
\end{alignat}
\end{subequations}
 It has been shown in \cite{Hou_Li_2005, Chae_2005} that the system $P^0_{\nu,0}$, in the case of whole space $\nR^2$, admits a unique global solution provided the initial data $(\bu_0, \theta_0) \in H^m(\nR^2)\times H^{m}(\nR^2)$ with $m\geq 3$,  $m$ an integer.   In fact, in \cite{Hou_Li_2005}, the authors only required $(\bu_0, \theta_0) \in H^m(\nR^2)\times H^{m-1}(\nR^2)$ with $m\geq 3$.  In \cite{Chae_2005}, it is shown that a Beale-Kato-Majda-type criterion is satisfied for the partially viscous system and therefore the system is globally well-posed.  In \cite{Chae_2005}, it is shown that the problems $P^0_{0,\kappa}$ and $P^0_{\nu,0}$ both admit a unique global solution provided the initial data $(\bu_0, \theta_0) \in H^m(\nR^2)\times H^{m}(\nR^2)$ with $m\geq 3$.  Similar results are shown in \cite{Hou_Li_2005} for $P^0_{0,\kappa}$ but with initial data $(\bu_0, \theta_0) \in H^m(\nR^2)\times H^{m-1}(\nR^2)$.  Global well-posedness results for rough initial data (in Besov spaces) is established in \cite{Hmidi_Keraani_2007}.

We establish in Section \ref{s:P_nu_zero_zero} the global well-posedness of $P^0_{\nu,0}$ in a periodic domain $\nT^2 = [0,1]^2$ assuming weaker initial data, namely,  $\bu_0\in H^1(\nT^2)$, (we always assume $\nabla\cdot\bu_0=0$) and $\theta_0\in L^2(\nT^2)$. Our key idea in proving the uniqueness result is by writing  $\theta = \Delta \xi$, with $\int_{\nT^2}\xi\,dx=0$, for some $\xi$, and then adapting the techniques of Yudovich in \cite{Yudovich_1963} (see also \cite{Majda_Bertozzi_2002}).   We note that the authors in \cite{Danchin_Paicu_2008_French} have shown the global well-posedness results in the whole space under a weaker assumption that $\bu_0,\theta_0 \in L^2(\nR^2)$.  The proof of their main results arise under the Besov and Lorentz space setting and involves the use of Littlewood-Paley decomposition and paradifferential calculus introduced by J. Bony \cite{Bony_1981}.    We include in this study global well-posedness results for the problem $P^0_{\nu,0}$ under a stronger assumption on the initial data, namely $\bu_0\in H^1(\nT^2)$, and $\theta_0 \in L^2(\nT^2)$ but using only elementary techniques in PDEs.  Although this particular result is not an improvement to that of \cite{Danchin_Paicu_2008_French}, we will see that applying our method in the case of anisotropic viscosity, we can establish an improvement to the global well-posedness results established in \cite{Danchin_Paicu_2008}.

In Section \ref{sec:aniso}, we then consider the case where the viscosity $\nu$ occurs in the horizontal direction only.  More precisely, assuming initial vorticity $\omega_0\in \sqrt{L}$ (defined below in \eqref{root_L_norm}), initial temperature (density)  $\theta_0 \in L^\infty(\nT^2)$, and  $\int_{\nT^2}\omega_0\,d\bx= \int_{\nT^2}\theta_0\,d\bx=0$, we establish global well-posedness for the following system, which we denote as $P^0_{\nu_x,0}$:
\begin{subequations}\label{bouss_aniso}
\begin{alignat}{2}
\label{bouss_aniso_mo}
\partial_t\bu + \sum_{j=1}^2\partial_j(u^j  \bu) &=\nu\partial_{1}^2 \bu-\nabla p + \theta \be_2,
\qquad&& \text{in }\nT^2\times[0,T],\\
\label{bouss_aniso_div}
\nabla \cdot \bu &=0,
\qquad&& \text{in }\nT^2\times[0,T],\\
\label{bouss_aniso_den}
\partial_t\theta + \nabla\cdot(\bu \theta) &=0,
\qquad&& \text{in }\nT^2\times[0,T],\\
\label{bouss_aniso_IC}
\bu(\bx,0)&=\bu_0(\bx),\quad\theta(\bx,0)=\theta_0(\bx),
\qquad&& \text{in }\nT^2.
\end{alignat}
\end{subequations}

Recently, in \cite{Danchin_Paicu_2008}, a global well-posedness result for the system $P^0_{\nu_x,0}$ (in the whole space $\nR^2$), under various regularity conditions on initial data, was successfully established.  More precisely, given that $\theta_0\in H^s(\nR^2)\cap L^\infty(\nR^2)$, with $s\in (1/2,1]$, $\bu_0\in H^1(\nR^2)$ and $\omega_0\in L^p(\nR^2)$ for all $2\leq p < \infty$, and such the $\omega_0$ satisfy
\begin{align}\label{root_L_norm}
\|\omega_0\|_{\sqrt{L}}:=\sup_{p\geq2}\frac{\|\omega_0\|_{L^p(\nR^2)}}{\sqrt{p-1}} <\infty,
\end{align}
the Boussinesq system \eqref{bouss_aniso} in the whole space with anisotropic viscosity admits a unique global regular solution.  The condition $\theta_0\in H^s$ with  $s\in(\frac{1}{2},1]$ was needed for establishing uniqueness in \cite{Danchin_Paicu_2008}.  We relax this condition in our current contribution.  We remark again that the main idea is to write $\theta = \triangle\xi$, and then proceed using the techniques of Yudovich \cite{Yudovich_1963} for the 2D incompressible Euler equations to prove uniqueness.  Furthermore, our method uses more elementary tools than those used in \cite{Danchin_Paicu_2008}.   It is worth mentioning that very recently, in \cite{Adhikari_Cao_Wu_2010}, the global regularity of classical solutions to the two-dimensional Boussinesq system in the case of vertical viscosity and vertical thermal diffusion was  established provided an additional extra thermal fractional diffusion of the form $(-\Delta)^\delta$ for $\delta>0$ is added.

Let us denote by $P_{\nu,\kappa}^{\alpha}$ the following system:

\begin{subequations}\label{bouss_v}
\begin{alignat}{2}
\label{bouss_v_mo}
-\alpha^2\triangle\partial_t\bu+\partial_t\bu + \sum_{j=1}^d\partial_j(u^j  \bu) &=-\nabla p + \theta \be_d + \nu\triangle\bu,
\qquad&& \text{in }\nT^d\times[0,T],\\
\label{bouss_v_div}
\nabla \cdot \bu &=0,
\qquad&& \text{in }\nT^d\times[0,T],\\
\label{bouss_v_den}
\partial_t\theta + \nabla\cdot(\bu \theta) &=\kappa\triangle\theta
\qquad&& \text{in }\nT^d\times[0,T],\\
\label{bouss_v_IC}
\bu(\bx,0)=\bu_0(\bx),\quad\theta(\bx,0)&=\theta_0(\bx),
\qquad&& \text{in }\nT^d.
\end{alignat}
\end{subequations}

In Section \ref{s:P_alpha_zero_zero}, we study in dimension $d=2$ the inviscid ($\nu = 0$), Voigt-$\alpha$ (with $\alpha>0$) regularized momentum equation, namely the system $P_{0,0}^{\alpha}$,  and in dimension $d=3$ the system $P_{0,\kappa}^{\alpha}$  (with $\kappa >0$).
%
In the case $d=2$  we establish global well-posedness results for the problem $P^\alpha_{0,0}$ given initial data $\bu_0\in H^2(\nT^2)$ with $\nabla\cdot\bu_0$, and $\theta_0 \in L^2(\nT^2)$.   This result also hold in the easier cases  $\kappa>0$ or $\nu>0$. In the case $d=3$, we require $\kappa>0$ to establish global well-posedness results.  We show that the problem $P^\alpha_{0,\kappa}$ with given initial data $\bu_0\in H^3(\nT^3)$ with $\nabla\cdot\bu_0$, and $\theta_0 \in L^\infty(\nT^3)$ is well-posed globally in time.   This result also hold in the easier case $\nu>0$. Observe that the system $P^\alpha_{0,0}$ formally coincides with the inviscid, non-diffusive Boussinesq equations when  $\alpha = 0$.  This type of inviscid $\alpha$-regularization can be traced back to the work of Cao, {\it et. al.}~\cite{Cao_Lunasin_Titi_2006} who proposed the inviscid simplified Bardina model (studied in \cite{Layton_Lewandowski_2006}) as regularization of the 3D Euler equations. The model consists of the Euler equations with the term $-\alpha^2\Delta\partial_t \bu$ added to the momentum equation.  We refer to this term as the Voigt term, and we refer to equations with this additional term as Voigt-regularized equations.  The reason for this terminology is that if one adds the Voigt term to the Navier-Stokes equations, the resulting equations happen to coincide with equations governing certain visco-elastic fluids known as Kelvin-Voigt fluids, which were first introduced and studied in the context of the 3D Navier-Stokes equations by A.P. Oskolkov \cite{Oskolkov_1973, Oskolkov_1982}, and were studied later in \cite{Kalantarov_1986}.  These equations are known as the Navier-Stokes-Voigt equations.  They were first  proposed in \cite{Cao_Lunasin_Titi_2006} as a regularization for either the Navier-Stokes (for $\nu>0$) or Euler (for $\nu=0$) equations, for small values of the regularization parameter $\alpha$.

 We briefly discuss the merits of the Navier-Stokes-Voigt equations, as they are a special case of the Boussinesq-Voigt equations.  Voigt-regularizations of parabolic equations are a special case of pseudoparabolic equations, that is, equations of the form $Mu_t+Nu=f$, where $M$ and $N$ are (possibly non-linear, or even non-local) operators.  For more about pseudoparabolic equations, see, e.g., \cite{DiBenedetto_Showalter_1981,Peszynska_Showalter_Yi_2009,Showalter_1975_nonlin,Showalter_1975_Sobolev2,Showalter_1972_rep,Carroll_Showalter_1976,Showalter_1970_SG,Showalter_1970_odd,Bohm_1992}.  Whether in the presence of either periodic boundary conditions or physical boundary conditions (under the assumption of the no-slip boundary conditions $u|_{\partial\Omega}=0$), the Navier-Stokes-Voigt equations enjoy global well-posedness, even in three-dimensions), as it has been pointed out in \cite{Cao_Lunasin_Titi_2006}.  The Euler-Voigt equations enjoy global well-posedness in the case of periodic boundary conditions (see, e.g., \cite{Cao_Lunasin_Titi_2006,Larios_Titi_2009}).
It is worth mentioning that the long-term dynamics and estimates for the global attractor, and the Gevrey regularity of solutions on the global attractor, of the three-dimensional Navier-Stokes-Voigt model were studied in \cite{Kalantarov_Titi_2009} and \cite{Kalantarov_Levant_Titi_2009}, respectively.  Moreover, it was shown recently in \cite{Ramos_Titi_2010} that the statistical solutions (i.e., invariant probability measures) of the three-dimensional Navier-Stokes-Voigt equations converge, in a suitable sense, to a corresponding statistical solution (invariant probability measure) of the three-dimensional Navier-Stokes equations.

In the context of numerical computations, the Navier-Stokes-Voigt system appears to have less stiffness than the Navier-Stokes system (see, e.g., \cite{Ebrahimi_Holst_Lunasin_2009, Levant_Ramos_Titi_2009}).  In \cite{Levant_Ramos_Titi_2009}, the statistical properties of the Navier-Stokes-Voigt model were investigated numerically in the context of the Sabra shell phenomenological model of turbulence and were compared with the corresponding Navier-Stokes shell model.

Due to its simplicity, the Voigt $\alpha$-regularization is also well-suited to being applied to other hydrodynamic models, such as the two-dimensional surface quasi-geostrophic equations, demonstrated in \cite{Khouider_Titi_2008}, and the three-dimensional magnetohydrodynamic (MHD) equations, demonstrated in \cite{Larios_Titi_2009}.  See also \cite{Ebrahimi_Holst_Lunasin_2009} for the application of Navier-Stokes-Voigt model in image inpainting. It is also worth mentioning that in the case of the inviscid Burgers equation, $u_t+uu_{x}=0$, this type of regularization leads to $-\alpha^2u_{xxt}+u_t+uu_x=0$, which is the well-known Benjamin-Bona-Mahony equation of water waves \cite{Benjamin_Bona_Mahony_1972}.  One goal of the present work is to lay some of the mathematical groundwork necessary to extend the Voigt regularization to the two-dimensional Boussinesq-equations, for the purpose of simplifying numerical simulations of the solutions to these equations.

It is worth mentioning that all the results reported here are equally valid in the presence of the Coriolis rotation term.

\section{Preliminaries}\label{sec:Pre}
In this section, we introduce some preliminary material and notations which are commonly used in the mathematical study of fluids, in particular in the study of the Navier-Stokes equations (NSE).  For a more detailed discussion of these topics, we refer to \cite{Constantin_Foias_1988,Temam_1995_Fun_Anal,Temam_2001_Th_Num,Foias_Manley_Rosa_Temam_2001}.

Let $\mathcal{F}$ be the set of all trigonometric polynomials with periodic domain $\nT^d:=[0,1]^d$.  We define the space of smooth functions which incorporates the divergence-free and zero-average condition to be
\[\mathcal{V}:=\set{\vphi\in\mathcal{F}^d:\nabla\cdot\vphi=0 \text{ and} \int_{\nT^d}\vphi\;dx=0}.\]
\noindent For the majority of this work, we take $d=2$.

We denote by $L^p$, $W^{s,p}$, $H^s\equiv W^{s,2}$, $C^{0,\gamma}$ the usual Lebesgue, Sobolev, and H\"older spaces, and define $H$ and $V$ to be the closures of $\mathcal{V}$ in $L^2$ and $H^1$ respectively.    We restrict ourselves to finding solutions whose average over the periodic box $\nT^d$ is zero.  Observe from the evolution equation of $\theta$ in the Boussinesq system of equations (as well as the Boussinesq-Voigt system of equations), if we assume that the average $\int_{\nT^d}\theta_0(x) dx=0$, then the average of $\int_{\nT^d}\theta(x,t)\;dx=0$ for all $t\geq 0$, and also$\int_{\nT^d}\bu(x,t)\;dx = 0$ for all $t\geq0$ provided $\int_{\nT^d}\bu_0(x) \;dx=0$.  Therefore, we can work in the spaces defined above consistently.   The notation $V^s:=H^s(\nT^d)\cap V$ will be convenient. When necessary, we write the components of a vector $\by$ as $y^j$, $j=1,2$.  We define the inner products on $H$ and $V$ respectively by
\[(\bu,\bv)=\sum_{i=1}^2\int_{\nT^d} u^iv^i\,dx
\sand
((\bu,\bv))=\sum_{i,j=1}^2\int_{\nT^d}\partial_ju^i\partial_jv^i\,dx,
\]
and the associated norms $|\bu|=(\bu,\bu)^{1/2}$, $\|\bu\|=((\bu,\bu))^{1/2}$.  (We use these notations indiscriminately for both scalars and vectors, which should not be a source of confusion). Note that $((\cdot,\cdot))$ is a norm due to the Poincar\'e inequality, \eqref{poincare}, below.
We denote by $V'$ the dual space of $V$.  The action of $V'$ on $V$ is denoted by $\ip{\cdot}{\cdot}\equiv \ip{\cdot}{\cdot}_{V'}$.  Note that we have the continuous embeddings
\begin{equation}\label{embed}
 V\hookrightarrow H\hookrightarrow V'.
\end{equation}
Moreover, by the Rellich-Kondrachov Compactness Theorem (see, e.g., \cite{Evans_1998,Adams_Fournier_2003}), these embeddings are compact.

Following \cite{Danchin_Paicu_2008}, we define the spaces
\begin{align*}
   \sqrt{L}:=\set{w\big|\|w\|_{\sqrt{L}}<\infty},
\end{align*}
where $\|\cdot\|_{\sqrt{L}}$ is defined by \eqref{root_L_norm}.  This space arises naturally, due to the following inequality, proven in \cite{Lieb_Loss_2001} (see also \cite{Danchin_Paicu_2008}),  which is valid in two dimensions:
\begin{align}\label{CZ_est}
   \|\bw\|_p\leq C\sqrt{p-1}\|\bw\|_{H^1},
\end{align}
for all $\bw\in H^1(\nT^2)$, for any $p\in[2,\infty)$, and where we denote by $\|\cdot\|_p$ the usual $L^{p}$ norm.  Note that clearly $L^\infty\subset\sqrt{L}\subset L^p$ for every $p\in[2,\infty)$.  We also recall the well-known elliptic estimate, due to the Biot-Savart law for an incompressible vector field $\bu$, satisfying $\nabla\cdot \bu=0$, and $\nabla\times\bu = \omega$, by means of the  Calder\'on-Zygmund theory for singular integrals:
\begin{equation}\label{Calderon}
\|\nabla\bu\|_p \leq C p \|\omega\|_p
\end{equation} for any $p\in (1,\infty)$ (see, e.g., \cite{Yudovich_1963}).

Let $Y$ be a Banach space.  We denote by $L^p([0,T],Y)$ (which we also denote as $L^p_TY_x$),  the space of (Bochner) measurable functions $t\mapsto w(t)$, where $w(t)\in Y$ for a.e. $t\in[0,T]$, such that the integral $\int_0^T\|w(t)\|_Y^p\,dt$ is finite (see, e.g., \cite{Adams_Fournier_2003}). A similar convention is used in the notation $C^k([0,T],X)$ for $k$-times differentiable functions of time on the interval $[0,T]$ with values in $Y$. Abusing notation slightly, we write $w(\cdot)$ for the map $t\mapsto w(t)$.  In the same vein, we often write the vector-valued function $w(\cdot,t)$ as $w(t)$ when $w$ is a function of $x$ and $t$.
We denote by $\dot{C}^\infty(\nT^2\times [0,T])$ the set of infinitely differentiable functions in the variable $x$ and $t$ which are periodic in $x$ with $\int_{\nT^2}\vphi(\cdot,t)\;dx=0$.  Similarly, we denote by $\dot{L}^p(\nT^2) = \set{\vphi\in L^p(\nT^2) : \int_{\nT^2}\vphi(x)\;dx=0}$.

We denote by $P_\sigma:\dot{L}^2\maps H$ the Leray-Helmholtz projection
operator and define the Stokes operator $A:=-P_\sigma\triangle$ with
domain $\mathcal{D}(A):=H^2\cap V$.  For $\vphi\in \mathcal{D}(A)$, we have the norm equivalence $|A\vphi|\cong\|\vphi\|_{H^2}$ (see, e.g., \cite{Temam_2001_Th_Num, Constantin_Foias_1988}).  In particular, the Stokes operator $A$ can be extended as a linear operator from $V$ into $V'$ associated with the bilinear form $((\bu,\bv))$,
$$\ip{A\bu}{\bv} = ((\bu,\bv)) \quad \mbox{ for all } \bv\in V.$$
It is known that $A^{-1}:H\maps \mathcal{D}(A) \hookrightarrow H$ is a
positive-definite, self-adjoint, compact operator from $H$ into itself, and therefore it has
an orthonormal basis of positive eigenvectors $\set{\bw_k}_{k=1}^\infty$ in $H$ corresponding to a non-increasing sequence of eigenvalues (see, e.g.,
\cite{Constantin_Foias_1988,Temam_1995_Fun_Anal}).  The vectors  $\set{\bw_k}_{k=1}^\infty$ are also the eigenvectors of $A$.  Since the corresponding eigenvalues of $A^{-1}$ can be ordered in a decreasing order,  we can label the eigenvalues $\lambda_k$ of $A$ so that
$0<\lambda_1\leq\lambda_2\leq\lambda_3\leq\cdots$.
Let $H_n:=\text{span}\set{\bw_1,\ldots,\bw_n}$, and let $P_n:H\maps H_n$ be the $L^2$ orthogonal projection onto $H_n$.   Notice that in the case of periodic boundary conditions in the torus $\nT^2$ we have $\lambda_1=(2\pi)^{-2}$.  We will abuse notation slightly and also use $P_n$ in the scalar case for the corresponding projection onto eigenfunctions of $-\triangle$, but this should not be a source of confusion.
Furthermore, in our case it is known that $A=-\triangle$ due to the periodic boundary conditions (see, e.g., \cite{Constantin_Foias_1988,Temam_1995_Fun_Anal}) and the eigenvectors $\bw_j$ are of the form ${\bf a}_{\bf k}e^{2\pi i{\bf k} \cdot {\bf x}}$, with ${\bf a}_{\bf k}\cdot {\bf k} =0$.

It will be convenient to use the standard notation of the Navier-Stokes bilinear term
\begin{equation}\label{Bdef}
 B(\bw_1,\bw_2):=P_\sigma\sum_{j=1}^d\partial_j(w_1^j\bw_2)
\end{equation}
for $\bw_1,\bw_2\in\mathcal{V}$.  We list some important properties of $B$ which can be found for example in \cite{Constantin_Foias_1988, Foias_Manley_Rosa_Temam_2001, Temam_1995_Fun_Anal, Temam_2001_Th_Num}.

\begin{lemma}\label{B:prop}
The operator $B$ defined in \eqref{Bdef} is a bilinear form which can be extended as a continuous map $B:V\times V\maps V'$ such that
 \begin{equation}
 \ip{B(\bw_1, \bw_2)}{\bw_3} = \int _{\nT^d} (\bw_1\cdot\nabla\bw_2)\cdot\bw_3\;dx,
 \end{equation}
 for every $\bw_1, \bw_2, \bw_3 \in \mathcal{V}$.
 satisfying the following properties:
 \begin{enumerate}[(i)]
  \item For $\bw_1$, $\bw_2$, $\bw_3\in V$,
\begin{equation}\label{B:Alt}
 \ip{B(\bw_1,\bw_2)}{\bw_3}_{V'}=-\ip{B(\bw_1,\bw_3)}{\bw_2}_{V'},
\end{equation}
and therefore
\begin{equation}\label{B:zero}
 \ip{B(\bw_1,\bw_2)}{\bw_2}_{V'}=0.
\end{equation}

\item For $\bw_1$, $\bw_2$, $\bw_3\in V$,
\begin{align}\label{B:ineq1}
   |\ip{B(\bw_1,\bw_2)}{\bw_3}_{V'}|
&\leq C|\bw_1|^{1/2}\|\bw_1\|^{1/2}\|\bw_2\||\bw_3|^{1/2}\|\bw_3\|^{1/2}\\
|\ip{B(\bw_1,\bw_2)}{\bw_3}_{V'}|
&\leq C|\bw_1|^{1/2}\|\bw_1\|^{1/2}|\bw_2|^{1/2}\|\bw_2\|^{1/2}\|\bw_3\|.
\end{align}
 \end{enumerate}
\end{lemma}

Let us define another very similar bilinear operator motivated by the transport term in the temperature equation.
\begin{equation}\label{B_theta_def}
\mathcal{B}(\bw,\psi):=\sum_{j=1}^d\partial_j(w^j \psi)
\end{equation}
for $\bw\in\mathcal{V}$ and $\psi \in \mathcal{F} $ with $\int_{\nT^d}\psi\;dx=0$.  We have the following similar properties for $\mathcal{B}$ which can be proven easily as in the proof of Lemma \ref{B:prop}.
\begin{lemma} \label{B_theta:prop}
The operator $\mathcal{B}$ defined in \eqref{B_theta_def} is a bilinear form which can be extended as a continuous map $\mathcal{B}:V\times H^1\maps H^{-1}$, such that
\begin{equation}\label{B_theta:def}
\ip{\mathcal{B}(\bw, \psi)}{\phi}_{H^{-1}} = - \int _{\nT^d} \bw\cdot\nabla\phi\;\psi\;dx,
\end{equation}
for every $\bw \in \mathcal{V}$ and $\phi,\psi\in \dot{C}^1$.   Moreover,
\begin{equation}\label{B_theta:Alt}
 \ip{\mathcal{B}(\bw,\psi)}{\phi}_{H^{-1}}=-\ip{\mathcal{B}(\bw,\phi)}{\psi}_{H^{-1}},
\end{equation}
and therefore
\begin{equation}\label{B_theta:zero}
 \ip{\mathcal{B}(\bw,\phi)}{\phi}_{H^{-1}}=0.
\end{equation}
Furthermore, $\mathcal{B}$ is also a bilinear form which can be extended as a continuous map $\mathcal{B}:\cD(A)\times L^2\maps H^{-1}$.
\end{lemma}
Here and below, $C, C_j$, etc. denote generic constants which may change from line to line.  $C_\alpha,C(\cdots)$, etc.  denote generic  constants which depend only upon the indicated parameters.  $K, K_j$, etc. denote constants which depend on norms of initial data, and also may vary from line to line.  Next, we recall that for an integrable function $f$ such that $\int_{\nT^2} f\;dx=0$, we have in two dimensions,
\begin{equation}\label{L4_to_H1}
\|f\|_{L^4}\leq|f|^{1/2}\|f\|^{1/2}.
\end{equation}
We also recall Agmon's inequality in two dimensions (see, e.g., \cite{Agmon_1965, Constantin_Foias_1988}).  For $\bw\in\mathcal{D}(A)$ we have
\begin{equation}\label{Agmon1/2}
 \|\bw\|_{L^\infty} \leq C|\bw|^{1/2}|A\bw|^{1/2}\:\:.
\end{equation}
 Furthermore, for all $\vphi\in V$, we have
the Poincar\'e inequality
\begin{equation}\label{poincare}
   \|\vphi\|_{L^2}\leq\lambda_1^{-1/2} \|\nabla\vphi\|_{L^2}.
\end{equation}
We will also make use of the following  inequality, valid in two dimensions, which is based on the Br\'ezis-Gallouet inequality, and which we prove in the appendix.  For every $\epsilon>0$, sufficiently small, and $\bw\in H^2(\nT^2)$,
\begin{align}\label{brezis}
   \|\bw\|_{L^\infty}\leq C\pnt{\|\bw\|\epsilon^{-1/4}+ |A\bw|e^{-1/\epsilon^{1/4}}},
\end{align}
where $C$ is independent of $\epsilon$.
Finally, we note a result of deRham \cite{Wang_1993, Temam_2001_Th_Num}, which states that if $g$ is a locally integrable function (or more generally, a distribution), we have
\begin{equation}\label{deRham}
 g =\nabla p \text{ for some distribution $p$ iff } \ip{g}{\bw}=0\quad\text{for all }
  \bw\in\mathcal{V},
\end{equation}
which one uses to recover the pressure.
\section{Global Well-posedness Results for the Viscous and Non-diffusive Boussinesq Equations. \texorpdfstring{($P^0_{\nu,0}$)}{}} \label{s:P_nu_zero_zero}

Let us first define the weak formulation of problem $P_{\nu,\kappa}^0$ in $\nT^2 \times [0,T]$.  By choosing a suitable phase space which incorporates the divergence free condition of the Boussinesq equations, we can eliminate the pressure from the equation, as is standard in the theory of the Navier-Stokes equations.  Consider the scalar test functions $\vphi(x,t) \in  \dot{C}^\infty(\nT^2 \times [0,T])$, such that $\vphi(x,T) =0$; and the vector test functions $\Phi(x,t)\in  [ \dot{C}^\infty(\nT^2 \times [0,T])]^2$ such that $\nabla\cdot\Phi(\cdot,t) =0$ and $\Phi(x,T) =0$.  Then the weak formulation of problem $P_{\nu,\kappa}^0$ in $\nT^2 \times [0,T]$ (and similarly of problem $P_{\nu,0}^0$, when $\kappa=0$, in $\nT^2 \times [0,T]$) is written as follows:

\begin{subequations}\label{bouss_wk}
\begin{align}
&\quad\notag
-\int_0^T(\bu(s), \Phi'(s))\,ds
+\nu\int_0^T((\bu(s), \Phi(s)  ))\,ds
+ \sum_{j=1}^2\int_0^T(u_j\bu,\partial_j\Phi) \,ds
 \\&\label{bouss_wk_mo}
=
(\bu_0(x),\Phi(x,0) )
+\int_0^T(\theta(s)\be_2,\Phi(s))\,ds
\\ & \notag \\
&\quad\notag
-\int_0^T(\theta(s),\vphi'(s))\,ds
+\label{bouss_wk_den}
\int_0^T(\bu\theta,\nabla\vphi)\,ds
+\kappa\int_0^T((\theta(s), \vphi(s)  ))\,ds
\\&=(\theta_0(x),\vphi(x,0)).
\end{align}
\end{subequations}

\begin{remark}\label{test_fcns_are_trig_polys}
   Note that it will become clear later that \eqref{bouss_wk} will hold for a larger class of test functions, and consequently it will be sufficient to consider only test functions of the form
\begin{subequations}\label{test_fcns}
\begin{align}
\label{test_vect}
\Phi(x,t) &= \Gamma_{\bm}(t) e^{2\pi i\bm\cdot \bx},
\text{ with } \Gamma_{\bm} \in [C^\infty([0,T])]^2\text{ and }\bm\cdot\Gamma_{\bm}(t)=0,
\intertext{and}
\label{test_scal}
\vphi(x,t) &= \chi_{\bm}(t) e^{2\pi i\bm\cdot \bx},
\text{ with }\chi_{\bm} \in C^\infty([0,T]),
\end{align}
\end{subequations}
for $\bm\in(\nZ\backslash \{0\})^2$, since such functions form a basis for the corresponding larger spaces of test functions.  
\end{remark}
In the two-dimensional case, the global well-posedness of system $P_{\nu,\kappa}^0$ in \eqref{bouss}, that is, in the case $\kappa>0$, $\nu>0$, is well-known, and can be proved in a similar manner following the work of \cite{Foias_Manley_Temam_1987} (see also \cite{Temam_1997_IDDS,Cannon_DiBenedetto_1980}).
We have the following existence and uniqueness results for the system $P_{\nu,\kappa}^0$, which will be used to prove the existence of weak solutions for the system $P_{\nu,0}^0$.
From here on, we only work on spaces of functions which are periodic and with spatial average zero.  Therefore, to simplify notation, we write $\dot{L}^2$ as $L^2$, $\dot{C}^k$ as $C^k$, etc.

\begin{theorem}\label{thm:diffusion}
   Let $T>0$, $\nu>0$ be fixed but arbitrary. Then, the following results hold:
   \begin{enumerate}[(i)]
   \item  If $\bu_0\in H$, $\theta_0\in L^2$ then for each $\kappa>0$, \eqref{bouss} has a unique solution $(\bu_\kappa,\theta_\kappa)$ in the sense of \eqref{bouss_wk} such that
   $\bu_\kappa\in C([0,T],H)\cap L^2([0,T],V)$,
   $\theta_\kappa\in C_w([0,T],L^2)$.  Furthermore, there exists a constant $K_0 > 0$ independent of $\kappa$ such that the following bounds hold:
   $\|\bu_\kappa\|_{L^2([0,T],V)}\leq K_0$,
   $\|\bu_\kappa\|_{L^\infty([0,T],H)}\leq K_0$,
   $\|\frac{d}{dt}\bu_\kappa\|_{L^2([0,T],V')}\leq K_0$,
   $\|\theta_\kappa\|_{L^\infty([0,T],L^2)} \leq |\theta_0|$, $\|\frac{d}{dt}\theta_\kappa\|_{L^2([0,T],H^{-2})} \leq K_0$ and
   $\sqrt{\kappa}\|\theta_\kappa\|_{L^2([0,T],H^1)}\leq K_0$.
   \item If the initial data $\bu_0\in V$ and $\theta_0\in L^2$,  then the solution $u_\kappa\in C([0,T], V)\cap L^2([0,T],\cD(A))$  and we also have the following bounds:
   $\|\bu_\kappa\|_{L^2([0,T],\mathcal{D}(A))}\leq K_0$,
   $\|\bu_\kappa\|_{L^\infty([0,T],V)}\leq K_0$,  $\|\frac{d}{dt}\bu_\kappa\|_{L^2([0,T],H)} \leq K_0$ and $\|\frac{d}{dt}\theta_\kappa\|_{L^2([0,T],H^{-1})} \leq K_0$.
   \item   If $\theta_0\in L^\infty$ and $\bu_0\in H$, then
   $\|\theta_\kappa\|_{L^\infty([0,T],L^\infty)}\leq \|\theta_0\|_\infty$.
   \item  If $u_0 \in H^3$ and $\theta_0 \in H^2$ then  for each $\kappa>0$, \eqref{bouss} has a unique solution $u_\kappa\in C([0,T],H^3)\cap L^2([0,T],H^4)$ and  $\theta_\kappa \in C([0,T],H^2)\cap L^2([0,T],H^3)$.
\end{enumerate}
\end{theorem}

\begin{proof}
Parts (i) and (ii) are essentially proven in \cite{Cannon_DiBenedetto_1980,Foias_Manley_Temam_1987, Temam_1997_IDDS} following the classical theory of Navier-Stokes equations.  The uniform bounds in part (ii) will be established explicitly in the later proofs when called for.  Part (iii) can be proven using maximum principle and is proven for example in \cite{Cannon_DiBenedetto_1980,Temam_1997_IDDS}. An explicit proof of this theorem will also be provided below.  Part (iv) can be proved using basic energy estimates and Gr\"onwall's inequality again following the classical theory of the Navier-Stokes equations.
\end{proof}
For the current study, we now define what we mean by weak solutions and  strong solutions for the viscous non-diffusive Boussinesq equations  ($P^0_{\nu,0}$).  We then state and prove our main results.

\begin{definition}[Weak solution]\label{def:weak}
   Let $T>0$.  Suppose $\bu_0\in H$ and $\theta_0\in L^2$.  We say that $(\bu,\theta)$ is a \textit{weak solution} to $P_{\nu,0}^0$ (that is, \eqref{bouss_nd} with $\kappa=0$) on the interval $[0,T]$, if $(\bu,\theta)$ satisfies
   the weak formulation \eqref{bouss_wk} (with $\kappa=0$), and
   $\bu\in L^2([0,T],V)\cap C([0,T],H)$,  $\pd{\bu}{t}\in L^1([0,T],V')$, with
   $\theta\in  C([0,T],L^2)$ and $\pd{\theta}{t}\in L^1([0,T],H^{-2})$.
 \end{definition}
\begin{definition}[Strong solution]\label{def:strong}
   Let $T>0$.  Suppose $\theta_0\in L^2$ and $\bu_0\in V$.  We say that $(\bu,\theta)$ is a \textit{strong solution} to $P_{\nu,0}^0$ (that is, \eqref{bouss_nd} with $\kappa=0$) on the interval $[0,T]$, if it is a weak solution in the sense of Definition \ref{def:weak}, and furthermore, $\bu\in L^2([0,T],\cD(A))\cap C([0,T],V)$, $\pd{\bu}{t}\in L^1([0,T],H)$, and
   $\pd{\theta}{t}\in L^1([0,T],H^{-1})$.
\end{definition}


We now state and prove our main results regarding global existence of weak and strong solutions to problem $P_{\nu,0}^0$.

\begin{theorem}[Existence of weak solutions]\label{exist_weak_visc}
   Let $T>0$ be given.  Let $\bu_0\in H$ and $\theta_0\in L^2$.  Then there exists a weak solution of \eqref{bouss} on the interval $[0,T]$.  Furthermore, system \eqref{bouss_wk} with $\kappa=0$ is equivalent to the functional form
  \begin{subequations}\label{e:functional_nu}
  \begin{align}\label{e:funct_1}
  \pd{\bu}{t} + \nu A\bu + B(\bu,\bu) &= P_\sigma(\theta\be_2)\quad \mbox{in}\quad L^2([0,T], V')\quad
  \mbox{and}\\\label{e:funct_2}
   \pd{\theta}{t} + \mathcal{B}(\bu,\theta) &= 0\quad \mbox{in}\quad L^2([0,T], H^{-2}).
   \end{align}
   \end{subequations}
   Moreover, if we assume $\theta_0\in L^\infty$, then  $\theta\in L^\infty([0,T], L^\infty)$.
\end{theorem}

\begin{proof}
Our method of proof involves passing to the limit of the weak solution of  \eqref{bouss} as $\kappa\maps 0$, that is, we consider $\kappa>0$ to be a regularization parameter to system \eqref{bouss_nd}.   Without loss of generality, we can assume $0<\kappa < 1$.  In accordance with Remark \ref{test_fcns_are_trig_polys} and Definition \ref{def:weak}, we only consider test functions of the form \eqref{test_fcns}.

We will show that the weak formulation
\begin{subequations}\label{bouss_kappa_weak}
\begin{align}
&\quad\notag
-\int_0^T(\bu_\kappa(s), \Gamma_{\bm}'(s) e^{2\pi i\bm\cdot \bx})\,ds
+\nu\int_0^T((\bu_\kappa(s), \Gamma_{\bm}(s) e^{2\pi i\bm\cdot \bx}      ))\,ds
 \\&\quad \notag
 +\sum_{j=1}^2\int_0^T(u^j_\kappa\bu_\kappa, \Gamma_{\bm}(s) \partial_j e^{2\pi i\bm\cdot \bx} )\,ds
\\&
\label{bouss_kappa_wk_mo}=
(\bu_{0},\Gamma_{\bm}(0) e^{2\pi i\bm\cdot \bx} )
+\int_0^T(\theta_\kappa(s)\be_2,\Gamma_{\bm}(s) e^{2\pi i\bm\cdot \bx} )\,ds
\\ & \notag \\
&
-\int_0^T(\theta_\kappa(s),e^{2\pi i\bm\cdot \bx} )\chi_{\bm}'(s)\,ds
+\notag
\int_0^T(\bu_\kappa(s)\theta_\kappa(s),\nabla e^{2\pi i\bm\cdot \bx} \chi_{\bm}(s))\,ds
\\&\label{bouss_kappa_wk_den}
  + \kappa \int_0^T((\theta_\kappa(s), e^{2\pi i \bm\cdot \bx}\chi_{\bm}(s)))\;ds=(\theta_{0},e^{2\pi i\bm\cdot \bx})\chi_{\bm}(0)
\end{align}
\end{subequations}
converges to the  weak formulation of $P^0_{\nu,0}$  (see \eqref{bouss_wk} with $\kappa=0$) as $\kappa\maps 0$.  After passing to the limit in the system we then show that the limiting functions satisfy the aforementioned regularity properties.  We proceed with the following steps.

\begin{list}{}{\leftmargin=0em}
\item {\em {\bf Step 1: } Using compactness arguments to prove convergence of a subsequence.}

From Theorem \ref{thm:diffusion}, in particular from the uniform bounds (with respect to $\kappa$) of $\bu_\kappa$, $\pd{\bu_\kappa}{t}$, $\theta_\kappa$ and $\pd{\theta_\kappa}{t}$ in the corresponding norms, one can use the Banach-Alaoglu Theorem and the Aubin Compactness theorem (see, e.g., \cite[Lemma 8.2]{Constantin_Foias_1988} or \cite{Temam_2001_Th_Num})  to justify that one can extract a subsequence of $(\bu_\kappa,\theta_\kappa)$ (which we still write as $(\bu_\kappa,\theta_\kappa)$) as $\kappa\maps 0$  and elements $\bu$ and $\theta$, such that
\begin{subequations}\label{wk_conv}
\begin{align}
\label{st_u_L2H}
\bu_\kappa&\maps\bu \quad\text{strongly in }L^2([0,T],H),\\
\label{wk_u_L2V}
\bu_\kappa&\rightharpoonup\bu \quad\text{weakly in }L^2([0,T],V)\text{ and weak-\tac \;in }L^\infty([0,T],H),\\
\label{wk_du_dt}
\pd{\bu_\kappa}{t}&\rightharpoonup\pd{\bu}{t} \quad\text{weakly in }L^2([0,T],V'),\\
\label{wk_theta_LiH}
\theta_\kappa&\rightharpoonup\theta \quad\text{weakly in }L^2([0,T],L^2)\text{ and weak-\tac \;in }L^\infty([0,T],L^2), \\
\label{wk_dtheta_dt}
\pd{\theta_\kappa}{t}&\rightharpoonup\pd{\theta}{t} \quad\text{weakly in }L^2([0,T],H^{-2}).
\end{align}
\end{subequations}\\

\item {\em {\bf Step 2:}  Passing to the limit in the system.}

The results from Step 1 imply that  for the linear terms in \eqref{bouss_kappa_weak}, we have, by the weak convergence in \eqref{wk_u_L2V} and  \eqref{wk_theta_LiH}, as $\kappa\maps 0$,
\begin{align*}
\int_0^T(\bu_\kappa(s), \Gamma'_{\bm}(s) e^{2\pi i\bm\cdot \bx})\,ds
  &\maps
  \int_0^T(\bu(s), \Gamma'_{\bm}(s) e^{2\pi i\bm\cdot \bx})\,ds
   ,\\
\nu\int_0^T((\bu_\kappa(s), \Gamma_{\bm}(s) e^{2\pi i\bm\cdot \bx}      ))\,ds
&\maps
\nu\int_0^T((\bu(s), \Gamma_{\bm}(s) e^{2\pi i\bm\cdot \bx}      ))\,ds
,\\
\int_0^T(\theta_\kappa(s)\be_2,\Gamma_{\bm}(s) e^{2\pi i\bm\cdot \bx} )\,ds
&\maps
\int_0^T(\theta(s)\be_2,\Gamma_{\bm}(s) e^{2\pi i\bm\cdot \bx} )\,ds
,\\
\int_0^T(\theta_\kappa(s),e^{2\pi i\bm\cdot\bx})\chi_{\bm}'(s)\,ds
&\maps
\int_0^T(\theta(s),e^{2\pi i\bm\cdot\bx})\chi_{\bm}'(s)\,ds
,\\
\kappa\abs{\int_0^T((\theta_\kappa(s), e^{2\pi i\bm\cdot \bx}\chi_{\bm}(s)))\;ds} & \leq C\sqrt{\kappa} (\sqrt{\kappa}\|\theta_\kappa\|_{L^2([0,T], H^1)})\leq CK_0\sqrt{\kappa}
\maps
0
\end{align*}

\bigskip

It remains to show the convergence of the remaining non-linear terms.  Let
\begin{align*}
   I(\kappa)&:=\sum_{j=1}^2\int_0^T(u^j_\kappa\bu_\kappa, \Gamma_{\bm}(s)\partial_j e^{2\pi i\bm\cdot \bx} )\,ds-
   \sum_{j=1}^2\int_0^T(u^j\bu, \Gamma_{\bm}(s) \partial_j e^{2\pi i\bm\cdot \bx} )\,ds
   \\
   J(\kappa)&:=\int_0^T\adv{\bu_\kappa(s)\theta_\kappa(s),\chi_{\bm}(s)\nabla e^{2\pi i\bm\cdot \bx}}\,ds-\int_0^T\adv{\bu(s)\theta(s),\chi_{\bm}(s)\nabla e^{2\pi i\bm\cdot \bx}}\,ds
\end{align*}
The convergence $I(\kappa)\maps 0$ as $\kappa\maps0$ is standard in the theory of the Navier-Stokes equations, thanks to \eqref{st_u_L2H} and \eqref{wk_u_L2V} (see, e.g., \cite{Temam_2001_Th_Num, Constantin_Foias_1988}).  To show $J(\kappa)\maps0$ as $\kappa\maps0$, we write $J(\kappa)=J_1(\kappa)+J_2(\kappa)$, the definitions of which are given below.  We  have
\begin{align*}
   J_1(\kappa)
   &:=
   \int_0^T(( \bu_\kappa(s)-\bu(s) ) \theta_\kappa(s),\nabla e^{2\pi i\bm\cdot \bx})\chi_{\bm}(s)\,ds
     \maps0
\end{align*}
as $\kappa\maps0$, since $\bu_\kappa\maps\bu$ strongly in $L^2([0,T],H)$ and $\theta_\kappa\maps\theta$ weakly in $L^2([0,T],H)$.  For $J_2$, we have
\begin{align*}
   J_2(\kappa)&:= \int_0^T\left(\bu(s)(\theta_\kappa(s)-\theta(s)),\nabla e^{2\pi i\bm\cdot \bx}\right)\chi_{\bm}(s)\,ds
     \maps0
\end{align*}
thanks to the weak convergence in \eqref{wk_theta_LiH} and the fact that $\bu\in L^2([0,T],H)$.
Thus, $J(\kappa)=J_1(\kappa)+J_2(\kappa)\maps0$.  Hence, sending $\kappa\maps0$, we see that $\bu$ and $\theta$ satisfy \eqref{bouss_wk}.  \\

\item {\em {\bf Step 3: }Show that $\bu\in C([0,T],H)$, that $\theta\in C_w([0,T],L^2)$, and that in fact $\theta\in C([0,T],L^2)$ .}

The uniform bound with respect to $\kappa$ on the time derivative of $\bu_\kappa$  given in Theorem \ref{thm:diffusion} (ii) allows us to pass to an additional subsequence if necessary to find that $\frac{d\bu}{dt}\in L^2([0,T],V')$.  Since $\bu\in L^2([0,T],V)$ and $\frac{d\bu}{dt}\in L^2([0,T],V')$, following the standard theory of NSE, (see, e.g.  Theorem 7.2 of \cite{Robinson_2001}) we obtain that $\bu\in C([0,T],H)$.

Next, we would like to show that $\theta\in C_w([0,T], L^2)$.   This can be proven without difficulty using standard arguments.   For completeness and for use in the later section we present the proof here.  We follow similar arguments as in  \cite{Levermore_Oliver_Titi_1996, Temam_2001_Th_Num}.  We start by showing that the sequence of solutions $\{\theta_\kappa\}$ (as $\kappa \maps 0$) is relatively compact in $C_w([0,T], H)$.  By  the Arzela-Ascoli theorem, it suffices to show that (a) $\{\theta_\kappa(t)\}$ is a relatively compact set in the weak topology of $L^2([0,T],\nT^2)$
for a.e $t\geq 0$ and (b) for every $\phi\in L^2(\nT^2)$ the sequence $\{(\theta_\kappa,\phi)\}$ is equicontinuous in $C([0,T])$. Condition (a) follows from the uniform boundedness of $\theta_\kappa(t)$  in $L^2(\nT^2)$ for a.e. $t\geq 0$,  as stated in Theorem \ref{thm:diffusion} part (i).  Next, we show that condition (b) is satisfied.  Following classical arguments,  we start by assuming that $\phi$ is smooth, for example we can assume that $\phi$ is a trigonometric polynomial.  We have
\begin{equation}
\aligned \label{e:continuity_theta_kappa}
&\quad
|(\theta_\kappa(t_2),\phi) - (\theta_\kappa(t_1), \phi)|
\\&\leq
\abs{
\kappa\int_{t_1}^{t_2}((\theta_\kappa(t),\phi))\;dt
}
+\abs{
 \int_{t_1}^{t_2} (\nabla(\bu_\kappa\theta_\kappa)(t),\phi)\;dt
}
\\&\leq
\kappa  \int_{t_1}^{t_2}  \|\theta_\kappa(t)\||\nabla\phi|  \;dt + \int_{t_1}^{t_2} (\bu_\kappa(t)\theta_\kappa(t),\nabla\phi)\;dt
\\&\leq
C\kappa^{1/2}|t_2-t_1|^{1/2}\left(\kappa\int_{t_1}^{t_2}\|\theta_\kappa(t)\|^2\;dt\right)^{1/2}
\\&\quad+
\|\nabla\phi\|_\infty |t_2-t_1|^{1/4}\left( \int_{t_1}^{t_2}\|\bu_\kappa(t)\|^4_4 \;dt \right)^{1/4}\left(\int_{t_1}^{t_2}|\theta_\kappa(t)|^2\;dt\right)^{1/2}.
\endaligned
\end{equation}
Without loss of generality assume $0<\kappa<1$ and use \eqref{L4_to_H1}, one then obtains
\begin{equation}\label{t-small}
\aligned
|(\theta_\kappa(t_2),\phi) - (\theta_\kappa(t_1), \phi)|&\leq C |t_2-t_1|^{1/2} \left(\kappa\int_{t_1}^{t_2}\|\theta_\kappa(t)\|^2\;dt\right)^{1/2} \\
&\quad    + C|t_2-t_1|^{1/4} \int_{t_1}^{t_2}|\bu_\kappa(t)|^2\|\bu_\kappa(t)\|^2 \;dt.
\endaligned
\end{equation}
From Theorem \ref{thm:diffusion} part (i), since $\bu_\kappa$ is uniformly bounded with respect to $\kappa$ in
$L^\infty([0,T], H) \cap L^2([0,T], V)$ and  $\displaystyle\kappa\int_{0}^T \|\theta_\kappa(t)\|^2\;dt < K_0$, with $K_0$ independent of $\kappa$, we have that the set $\{(\theta_\kappa,\phi) \}$ is equicontinuous in $C([0,T])$.  We now extend this result for all test functions $\phi$ in $\dot{L}^2(\nT^2)$ using a simple density argument of trigonometric polynomials in $\dot{L}^2(\nT^2)$.   Let $\epsilon>0$.   We choose  a trigonometric polynomial $\phi_\epsilon$   such that $|\phi-\phi_\epsilon | < \frac{\epsilon}{3|\theta_0| + 1}$.  Then, we have
\begin{equation}\label{eq:above}
\aligned
|(\theta_\kappa(t_2),\phi) - (\theta_\kappa(t_1), \phi)| &= |( \theta_\kappa(t_2)- \theta_\kappa(t_1), \phi-\phi_\epsilon) +  ( \theta_\kappa(t_2)- \theta_\kappa(t_1),\phi_\epsilon)  |\\
&\leq |\phi-\phi_\epsilon|\left(|\theta_\kappa(t_2)| + |\theta_\kappa(t_1)|\right) + |( \theta_\kappa(t_2)- \theta_\kappa(t_1),\phi_\epsilon)    |.
\endaligned
\end{equation}

From the uniform $L^\infty([0,T],L^2)$ bound of $\theta_\kappa$, with respect to $\kappa$, we conclude that the first term on the right-hand side of \eqref{eq:above} is less than $\frac{2}{3}\epsilon$.    Choosing $|t_2-t_1|$ small enough in \eqref{t-small} we can make the second term on the right-hand side of \eqref{eq:above} to be less than $\epsilon/3$.  Thus, the whole expression can be made less than $\epsilon$.  This completes the proof that $\theta\in C_w([0,T], L^2)$.  Finally, as pointed out by the authors \cite{Danchin_Paicu_2008}, since $\theta$ is transported by the div-free velocity field $\bu\in L^2([0,T], V)$, we get in addition that  $\theta\in C([0,T], L^2)$,  (see, e.g. \cite{DiPerna_Lions_1989} ).

From these results, standard arguments from the theory of the Navier-Stokes equations (see, e.g., \cite{Constantin_Foias_1988,Temam_2001_Th_Num,Robinson_2001}) now show that the initial conditions are satisfied in the sense of Definition \ref{def:weak}.

\item {\em {\bf Step 4:}  Show that if $\theta_0\in L^\infty$ then $\theta \in L^\infty([0,T], L^\infty)$.}

Here we will use E. Hopf and G. Stampacchia technique which are very similar to those used in \cite{Foias_Manley_Temam_1987} (see also \cite{Kinderlehrer_Stampacchia_1980, Temam_1997_IDDS}), but we give the details here for the sake of completeness.
For any function $f\in H^1$, we use the standard notation $f^+:=\max\{f,0\}.$    It is a standard exercise to show that if $f\in H^1$, then $f^+\in H^1$.
Let $(\bu_\kappa,\theta_\kappa)$ be a solution of \eqref{bouss}, as given in Theorem \ref{thm:diffusion}.  Let us denote $\Theta_\kappa:=\theta_\kappa-\|\theta_0\|_{L^\infty}$.  Notice that $\Theta_\kappa$ satisfies the evolution equation \eqref{bouss_den} with $\theta$ replaced by $\Theta_\kappa$ and $\bu$ replaced by $\bu_\kappa$.  Thus we have  $(\Theta_\kappa)^+\in L^2([0,T],H^1)$.  Taking the action of \eqref{bouss_den} with $(\Theta_\kappa)^+$ yields
\begin{align*}
   \frac{1}{2}\frac{d}{dt}\|(\Theta_\kappa)^+\|_{L^2}^2
   &=
   -\kappa\int_{\nT^2}|\nabla(\Theta_\kappa)^+|^2\,d\bx
   +\int_{\nT^2}\bu_\kappa\Theta_\kappa\cdot \nabla\Theta_\kappa^+\,d\bx \\&=
   -\kappa\int_{\nT^2}|\nabla(\Theta_\kappa)^+|^2\,d\bx +\frac{1}{2}\int_{\nT^2} \bu_\kappa\cdot\nabla(\Theta^+)^2 \;d\bx
      \\&=
   -\kappa\int_{\nT^2}|\nabla(\Theta_\kappa)^+|^2\,d\bx
   \leq0,
\end{align*}
thanks to \eqref{bouss_div}.  Thus $\|(\Theta_\kappa)^+(t)\|_{L^2}\leq \|(\Theta_\kappa)^+(0)\|_{L^2}=0$ which implies that $(\theta_\kappa(\bx,t)- \|\theta_0\|_{L^\infty})^+\leq0$ a.e.
Similarly, one can show that $( \|\theta_0\|_{L^\infty}-\theta_\kappa(\bx,t))^+\geq 0$ a.e.  It now follows that $\|\theta_\kappa\|_{L^\infty([0,T],L^\infty)}\leq \|\theta_0\|_{L^\infty}$ for all $\kappa>0$.  Thus, we have $\theta_\kappa$ is bounded uniformly with respect to $\kappa$ in $L^\infty([0,T],L^\infty)$.  Therefore, it follows from the Banach-Alaoglu Theorem,  that there exists a subsequence of the previous subsequence which we also denote as $\theta_\kappa$ converging in weak-\tac\; topology of $L^\infty([0,T], L^\infty)$ to $\theta$ and satisfies the following bounds: $\displaystyle\|\theta\|_{L^\infty([0,T],L^\infty)}\leq\lim\inf_{\kappa\maps 0} \|\theta_\kappa\|_{L^\infty([0,T],L^\infty)}  \leq\|\theta_0\|_{L^\infty}<\infty$.

The equivalence of \eqref{bouss_wk} to the functional form \eqref{e:functional_nu} follows from the standard argument of NSE (see, e.g., \cite{Temam_2001_Th_Num}).
 \end{list}
\end{proof}

\begin{theorem}[Existence of strong solutions]\label{exist_strong_visc}
   Let $\bu_0\in V$ and $\theta_0\in L^2$.  Then there exists a strong solution of $P_{\nu,0}^0$. Furthermore, the functional equation \eqref{e:funct_1} now holds in $L^2([0,T],H)$, and \eqref{e:funct_2} now holds in $L^2([0,T], H^{-1})$.  \end{theorem}
\begin{proof}
   Since $\bu_0\in V$ and $\theta_0\in L^2$, we have by Theorem \ref{thm:diffusion} that there exists a solution $(\bu_\kappa,\theta_\kappa)$ of \eqref{bouss} with $\bu_\kappa\in L^\infty([0,T],V)\cap L^2([0,T],\mathcal{D}(A))$, and furthermore, $\frac{d}{dt}\bu_\kappa\in L^2([0,T],H)$.  In order to show that in fact the bounds on higher-order norms are independent of $\kappa$ (as stated in Theorem \ref{thm:diffusion} part (ii)), let us take the inner product of \eqref{bouss_nd_mo} with $A\bu_\kappa$.  Using  the Lions-Magenes lemma (see, e.g., \cite{Temam_2001_Th_Num}) to show that $\ip{\pd{\bu}{t}}{A\bu} = \frac{1}{2}\frac{d}{dt}\|\bu\|^2$ and the fact that $(B(\bu_\kappa,\bu_\kappa),A\bu_\kappa)=0$ due to the periodic boundary conditions, we have
   \begin{align*}
      \frac{1}{2}\frac{d}{dt}\|\bu_\kappa\|^2 +\nu|A\bu_\kappa|^2
      &=
      (\theta_\kappa \be_2,A\bu_\kappa)
      \leq
      |\theta_0||A\bu_\kappa|
      \leq
      \frac{1}{2\nu}|\theta_0|^2+\frac{\nu}{2}|A\bu_\kappa|^2.
   \end{align*}
Subtracting $\frac{\nu}{2}|A\bu_\kappa|^2$ and using Gr\"onwall's inequality yields
   \begin{align}\label{u_H2_bdd}
      \|\bu_\kappa(t)\|^2 +\nu\int_0^t|A\bu_\kappa|^2 \,ds
      &\leq
      \|\bu_0\|^2+\frac{1}{\nu}|\theta_0|^2t
      \leq
      \|\bu_0\|^2+\frac{1}{\nu}|\theta_0|^2T:=K_3.
   \end{align}
Thus $\bu_\kappa$ is bounded in $L^{\infty}([0,T],V)\cap L^2([0,T],\mathcal{D}(A))$ independently of $\kappa$.  Furthermore,
\begin{align*}
   \abs{\frac{d\bu_\kappa}{dt}}
   &=\sup_{|\bw|=1}\pair{B(\bu_\kappa,\bu_\kappa)}{\bw}
   +\nu\sup_{|\bw|=1}(A\bu_\kappa,\bw)
   +\sup_{|\bw|=1}(\theta_\kappa\be_2,\bw)
   \\&\leq
\sup_{|\bw|=1}|A\bu_\kappa|\|\bu_\kappa\||\bw|
   +\nu\sup_{|\bw|=1}|A\bu_\kappa||\bw|
   +\sup_{|\bw|=1}|\theta_\kappa||\bw|
    \\&\leq K_3|A\bu_\kappa|
   +\nu|A\bu_\kappa|
   +|\theta_0|
\end{align*}
Thus, $\frac{d}{dt}\bu_\kappa$ is bounded in $L^2([0,T],H)$ independently of $\kappa$ due to \eqref{u_H2_bdd}.
We also have,
\begin{align*}
   \norm{\frac{d\theta_{\kappa}}{dt}}_{H^{-1}}
   &=
   \sup_{\|\bw\|=1}\abs{\ip{\mathcal{B}(\bu_{\kappa},\theta_{\kappa})}{\bw}} +  \kappa\sup_{\|\bw\|=1}\abs{\ip{\nabla\theta_{\kappa}}{\nabla\bw}}\\
   &\leq
     \sup_{\|\bw\|=1}\|\bu_{\kappa}\|_{L^\infty}|\theta_{\kappa}|\|\bw\| +\kappa \sup_{\|\bw\|=1}\|\theta_\kappa\| \|w\|  \\
    &\leq  \|\bu_{\kappa}\|_{H^2}|\theta_0| + \sqrt{\kappa}\|\theta_\kappa\|,
\end{align*}
where we have used here the assumption that $0< \kappa < 1$.   Hence, from Theorem \ref{thm:diffusion}, we have that $\frac{d\theta_{\kappa}}{dt}$ is bounded in $L^2([0,T],H^{-1})$ independently of $\kappa$.
The above estimates allow us to use the Banach-Alaoglu Theorem and the Aubin Compactness Theorem (see, e.g., \cite{Temam_2001_Th_Num, Constantin_Foias_1988}) , as $\kappa\maps0$, to extract a further subsequence (extracted from the sequences in \eqref{st_u_L2H}, \eqref{wk_u_L2V} and \eqref{wk_theta_LiH}, and which we still label with a subscript $\kappa$, such that
\begin{subequations}\label{st_conv}
\begin{align}
\label{st_u_L2V}
\bu_\kappa&\maps\bu \quad\text{strongly in }L^2([0,T],V),
\\
\label{wk_u_L2DA}
\bu_\kappa&\rightharpoonup\bu \quad\text{weakly in }L^2([0,T],\cD(A))\text{ and weak-\tac\; in }L^\infty([0,T],V),\\
\label{wk_d_u_H}
\frac{d\bu_\kappa}{dt}&\rightharpoonup\frac{d\bu}{dt} \quad\text{weakly in }L^2([0,T],H),\\
\label{wk_d_theta}
\frac{d\theta_\kappa}{dt}&\rightharpoonup\frac{d\theta}{dt} \quad\text{weakly in }L^2([0,T],H^{-1}),
\end{align}
\end{subequations}
where the limit $\bu$ and $\theta$ are the same elements as in \eqref{wk_conv}, by the uniqueness of limits since the current topology are stronger than those in \eqref{wk_conv}. Furthermore,  since $\bu\in L^2([0,T],\cD(A))$ and $\frac{d\bu}{dt}\in L^2([0,T],H)$, following the standard theory of NSE, (see, e.g.  Theorem 7.2 of \cite{Robinson_2001}) we obtain that $\bu\in C([0,T],V)$.   Thus we have shown the existence of a strong solution as defined in Definition \ref{def:strong}.
\end{proof}

In the next theorem we will show the uniqueness of strong solutions. We note that in the work of \cite{Danchin_Paicu_2008_French}, global well-posedness in the case of whole plane $\nR^2$ was established with initial data $\bu_0$ and $\theta_0$ both only in $L^2$.  This optimal global well-posedness result was established using elegantly {\em a priori} estimates in Besov spaces for the heat equation and the transport equation.  Here we give an alternate proof which uses more elementary techniques but we require stronger initial data for the velocity field.  This will allow us to fix some basic ideas that we will use to get the optimal global well-posedness results for the anisotropic Boussinesq equations which we will present in the next section.

\begin{theorem}[Uniqueness of Strong Solutions of $P_{\nu,0}^0$]\label{uniqueness_visc}
   Let $T>0$.  Suppose $\theta_0\in L^2$ and $\bu_0\in V$.  Then there exists a unique strong solution $(\bu,\theta)$ to $P_{\nu,0}^0$.
\end{theorem}

\begin{proof}
   The existence of solutions satisfying the hypothesis is already given by Theorem \ref{exist_strong_visc}.  It remains to show the uniqueness.  Let $(\bu_\ell,\theta_\ell)$ be two strong solutions, $\ell=1,2$, and define $\xi_\ell:=\triangle^{-1}\theta_\ell$ such that $\int_{\nT^2}\xi_\ell\,d\bx=0$ on the interval $[0,T]$.  Write $\bud:=\bu_1-\bu_2$, and $\xid:=\xi_1-\xi_2$.   These quantities satisfy the functional equations
   \begin{subequations}\label{e:functional_nu_diff}
  \begin{align}\label{e:funct_1_diff}
  \pd{\bud}{t} + \nu A\bud + B(\bu_1,\bud) + B(\bud,\bu_2)&= P_\sigma(\triangle\xid\be_2)\quad \mbox{in}\quad L^2([0,T], H)\quad
  \mbox{and}\\\label{e:funct_2_diff}
   \pd{\triangle\xid}{t} + \mathcal{B}(\bud,\triangle\xi_1) + \mathcal{B}(\bu_2,\triangle\xid) &= 0\quad \mbox{in}\quad L^2([0,T], H^{-1}).
   \end{align}
   \end{subequations}
  Taking the inner product in $H$ of \eqref{e:funct_1_diff} with $\bud$ , and taking the action in  $H^{-1}$ of \eqref{e:funct_2_diff} on $\xid \in L^2([0,T], H^{2})$, we obtain, thanks to Lemmas \ref{B:prop} and \ref{B_theta:prop},
\begin{subequations}\label{bouss_diff_en}
\begin{align}
\label{bouss_diff_en_u}
   \frac{1}{2}\frac{d}{dt}|\bud|^2
   &+ \nu\|\bud\|^2
   =
  -(B(\bud,\bu_1), \bud)
       +(\triangle\xid\be_2,\bud),
   \\\label{bouss_diff_en_xi}
   \frac{1}{2}\frac{d}{dt}\|\xid\|^2
    &=
        (\bud\triangle\xi_1,\nabla\xid)
      -(\bu_2\triangle\xid,\nabla\xid).
\end{align}
\end{subequations}
In \eqref{bouss_diff_en_xi} we used the Lions-Magenes Lemma (see, e.g., \cite{Temam_2001_Th_Num})  to obtain $\frac{1}{2}\frac{d}{dt}\|\xid\|^2 = \ip{\frac{d\xid}{dt}}{\xid}$.
    Let $K=\max_{\ell=1,2}\set{\|\bu_\ell\|_{L^\infty([0,T], V)},\|\theta_\ell\|_{L^\infty([0,T], L^2)}}$.  From equation \eqref{bouss_diff_en_u}, \eqref{B:ineq1} and since $\bu_1\in L^\infty([0,T],V)$ we have
   \begin{align}
      \frac{1}{2}\frac{d}{dt}|\bud|^2  + \nu\|\bud\|^2
      &\notag
      \leq
      C|\bud|\|\bud\|\|\bu_1\|+\|\xid\|\|\bud\|
      \\&\leq\label{bouss_u_diff}
      \frac{K}{\nu}|\bud|^2+\frac{\nu}{6}\|\bud\|^2 + \frac{3}{2\nu}\|\xid\|^2+\frac{\nu}{6}\|\bud\|^2.
   \end{align}

    Next, let $\epsilon>0$ be given such that $\epsilon\ll1$.  For the equation \eqref{bouss_diff_en_xi}, we integrate by parts and use \eqref{B_theta:zero} to find

 \begin{align}
      \frac{1}{2}\frac{d}{dt}\|\xid\|^2
      &=\notag
      \pair{\bud\triangle\xi_1}{\nabla\xid}
      +\sum_{j=1}^2\pair{\partial_j\bu_2}{\nabla\xid\partial_j\xid}
      \\&\leq \notag
      \|\bud\|_{L^\infty}|\triangle\xi_1|\|\xid\|
      +\|\nabla \bu_2\|_{L^{2/\epsilon}}\|\xid\|\|\nabla\xid\|_{L^{2/(1-\epsilon)}}
      \\&\leq \notag
      K\pnt{\|\bud\|\epsilon^{-1/4}+|A\bud|e^{-1/\epsilon^{1/4}}}\|\xid\|
      +C\|\nabla \bu_2\|_{L^{2/\epsilon}}\|\xid\|\|\xid\|^{1-\epsilon}\|\xid\|_{H^2}^\epsilon
      \\&\leq \label{bouss_xi_diff}
      \frac{\nu}{6}\|\bud\|^2
      +C\pnt{\frac{K^2}{\nu}\epsilon^{-1/2}
      +1}\|\xid\|^2
      +|A\bud|^2e^{-2/\epsilon^{1/4}}
      \\&\quad \notag
      +K^{\epsilon}\epsilon^{-1/2}|A\bu_2|\|\xid\|^{2-\epsilon},
   \end{align}
   where we have used  \eqref{CZ_est}, \eqref{brezis}, and the interpolation inequality $\|\nabla\xid\|_{L^{2/(1-\epsilon)}}\leq C\|\xid\|^{1-\epsilon}\|\xid\|_{H^2}^\epsilon$, noting that $C$ is independent of $\epsilon$.

Next, we will use the fact that, $\xid(0)=0$ and $\bud(0)=0$, and also that $\|\xid(t)\|$ and $|\bud(t)|$ are continuous in time and thus there exist a $\tau>0$ such that  $\|\xid(t)\|<1$ and $|\bud(t)|<1$ for all $t\in[0,\tau]$.  Let $t^* = \sup\{\tau\in (0,T] : |\bud(t)|<1 \mbox{ and } \|\xid(t)\|<1 \mbox{ for  all } t\in[0,\tau)\}$.   Adding \eqref{bouss_u_diff} and \eqref{bouss_xi_diff} and rearranging, we have on $[0,t^*]$,
\begin{align}
&\quad\notag
   \frac{1}{2}\frac{d}{dt}\pnt{|\bud|^2+\|\xid\|^2 }+\frac{\nu}{2}\|\bud\|^2
   \\&\leq \notag
   K_\nu\pnt{1+\frac{1}{ \epsilon^{1/2}}}\pnt{|\bud|^2 + \|\xid\|^2}
      +|A\bud|^2e^{-2/\epsilon^{1/4}}
      +K^{\epsilon}\epsilon^{-1/2}|A\bu_2|\|\xid\|^{2-\epsilon}
   \\&\leq \label{less_than_one}
   K_\nu\pnt{1+\frac{1}{ \epsilon^{1/2}}+\frac{K^{\epsilon}}{\epsilon^{1/2}}|A\bu_2|}\pnt{|\bud|^2 + \|\xid\|^2}^{1-\epsilon}
      +|A\bud|^2e^{-2/\epsilon^{1/4}}.
\end{align}
 Let $\eta>0$, be arbitrary and let $z:=|\bud|^2+\|\xid\|^2+\eta$.  Dividing \eqref{less_than_one} by $z^{1-\epsilon}$, we find
\begin{align*}
   \frac{1}{\epsilon}\frac{d}{dt}z^{\epsilon}
   &\leq
   K_\nu\pnt{1+\frac{1}{ \epsilon^{1/2}}+\frac{K^{\epsilon}}{\epsilon^{1/2}}|A\bu_2|}
   +z^{\epsilon-1}|A\bud|^2e^{-2/\epsilon^{1/4}}
   \\&\leq
   K_\nu\pnt{1+\frac{1}{ \epsilon^{1/2}}+\frac{K^{\epsilon}}{\epsilon^{1/2}}|A\bu_2|}
   +(\eta)^{\epsilon-1}|A\bud|^2e^{-2/\epsilon^{1/4}},
\end{align*}
since $z\geq\eta$.  Integrating over $[0,t]$, for $t\in(0,t^*]$, we  find
\begin{align}\label{z_int}
  z(t)
   &\leq
   K_\nu^{1/\epsilon}\pnt{\epsilon T+\epsilon^{1/2}T+K^{\epsilon}\epsilon^{1/2}\int_0^T|A\bu_2(s)|\,ds}^{1/\epsilon}
   \\&\quad\nonumber
   +\epsilon^{1/\epsilon}(\eta)^{1-1/\epsilon}e^{-2\epsilon^{-5/4}}\pnt{\int_0^T|A\bud(s)|^2\,ds}^{1/\epsilon} + \eta^{1/\epsilon}.
\end{align}
Sending $\eta\maps 0$, we obtain
\begin{align}\label{e:E4}
|\bud(t)|^2 + \|\xid(t)\|^2 \leq K_\nu^{1/\epsilon} \left(\epsilon T + \epsilon^{1/2}T + cK^\epsilon\epsilon^{1/2} \int_0^T|A\bu_2(s)|^2\;ds\right)^{1/\epsilon}
\end{align}
for $t\in [0,t^*]$.  Taking the limit of \eqref{e:E4}, as $\epsilon\maps0$, we find that $\|\xid(t)\|=0$ and $|\bud(t)|=0$ on $[0,t^{*}]$.  In particular,  $|\bud(t^*)|^2= \|\xid(t^*)\|^2=0<1$.  Therefore, from the continuity of $|\bud(t)|^2$ and $\|\xid(t)\|^2$ and the definition of $t^*$, we conclude that $t^*= T$, otherwise we a contradiction to the definition of $t^*$.  Hence, $\bud(t)=0$ and $\xid(t)=0$ for all $t\in [0,T]$.
 \end{proof}
\section{Global Well-posedness Results for the non-diffusive Boussinesq Equations with Horizontal Viscosity \texorpdfstring{($P_{\nu_x,0}^0$)}{}}\label{sec:aniso}

We now consider the Boussinesq equations with anisotropic viscosity  as given in \eqref{bouss_aniso} ($P_{\nu_x,0}^0$).   We will establish here global well-posedness results under some not too restricted initial conditions.  In the first part of this section we will first define what we mean by weak solution to system \eqref{bouss_aniso} and then show its existence.   Then, under some additional requirements on initial data, we can show uniqueness.
To set additional notation, we denote the vorticity $\omega := \partial_1u^2-\partial_2 u^1$,  which satisfies the following equation
\begin{align}\label{bouss_aniso_vort}
 \partial_t\omega + \nabla\cdot(\omega\bu) -\nu\partial_1^2\omega&= \partial_1\theta.
 \end{align}

The best global well-posedness result we are aware of for problem \eqref{bouss_aniso} in the case of the whole plane $\nR^2$ is stated in  following theorem, established in \cite{Danchin_Paicu_2008}.
\begin{theorem}[Danchin and Paicu,\cite{Danchin_Paicu_2008}]\label{t:Paicu_exist}
Let $\Omega = \nR^2$.  Suppose $\theta_0\in L^2\cap L^\infty$ , and $\bu_0\in V$ with $\omega_0\in \sqrt{L}$.  Then system \eqref{bouss_aniso} admits a global solution $(\bu,\theta)$ such that
$\theta\in C_B([0,\infty);L^2)\cap C_w([0,\infty);L^\infty)\cap L^\infty([0,\infty),L^\infty)$
and $\bu \in C_w([0,\infty);H^1)$,
$\bu\cdot\be_2\in L^2_{\text{loc}}([0,\infty);H^2)$,
$\omega\in L^\infty_{\text{loc}}([0,\infty),\sqrt{L})$,
$\nabla \bu\in  L^2_{\text{loc}}([0,\infty),\sqrt{L})$.
If in addition $\theta_0\in H^s$ for some $s\in (0, 1]$, then $\theta\in C([0,\infty);H^{s-\epsilon})$
 for all $\epsilon > 0$.   Finally, if $s > 1/2$, then the solution is unique.
\end{theorem}
In the present work, we improve the above result by weakening the requirements on the initial data needed for the uniqueness portion of the theorem.  To begin with, we weaken the notion of solution by making the following definition.

\begin{definition}[Weak Solutions for the Anisotropic Case]\label{def:soln_aniso}
   Let $T>0$.  Let $\theta_0\in L^2$, $\omega_0\equiv\nabla^\perp\cdot\bu_0\in L^2$.  We say that $(\bu,\theta)$ is a weak solution to \eqref{bouss_aniso} on the interval $[0,T]$ if $\omega\in L^\infty([0,T];L^2)\cap C_w([0,T];L^2)$ and $\theta\in L^\infty([0,T];L^2)\cap C_w([0,T];L^2),$  $u^2\in L^2([0,T],H^2)$, $\pd{\bu}{t}\in L^1([0,T], V')$, $\pd{\theta}{t}\in L^1([0,T],H^{-2})$ and also $(\bu,\theta)$ satisfies \eqref{bouss_aniso} in the weak sense; that is, for any $\Phi$, $\vphi$, chosen as in \eqref{test_fcns}, it holds that
\begin{subequations}\label{bouss_aniso_weak_form}
\begin{align}
   &\quad\notag
   -\int_0^T(\bu(s),\Phi'(s))\,ds+\nu\int_0^T(\partial_1\bu(s),\partial_1\Phi(s))\,ds
   +\sum_{j=1}^2\int_0^T(u^j\bu,\partial_j\Phi)\,ds
   \\&\label{}
    = (\bu_0,\Phi(0))+\int_0^T(\theta(s)\be_2,\Phi(s))\,ds
   \\\notag\\&
   -\int_0^T(\theta(s),\vphi'(s))\,ds + \int_0^T(\theta\bu,\nabla\vphi)\,ds = (\theta_0,\vphi(0)),
\end{align}
\end{subequations}
where $'\equiv \pd{}{s}$.
\end{definition}
\begin{remark}
Again following standard arguments as in the theory of NSE \cite{Temam_2001_Th_Num} one can show that the above system is equivalent to the functional form
\begin{subequations}\label{e:functional_nu_horizontal}
  \begin{align}\label{e:funct_1_horizontal}
  \pd{\bu}{t} + \nu \partial_{1}^2\bu + B(\bu,\bu) &= P_\sigma(\theta\be_2)\quad \mbox{in}\quad L^2([0,T], V')\quad
  \mbox{and}\\\label{e:funct_2_horizontal}
   \pd{\theta}{t} + \mathcal{B}(\bu,\theta) &= 0\quad \mbox{in}\quad L^2([0,T], H^{-2}).
   \end{align}
   \end{subequations}
\end{remark}
We now state and prove our main results for the system  $\eqref{bouss_aniso}$ ($P_{\nu_x,0}^0$).  The global existence and regularity results will be stated in the theorem below and the uniqueness theorem will follow.
\begin{theorem}[Global Existence and Regularity]\label{EU_aniso}
   Let $T>0$ be given.  Let $\theta_0\in L^2$ and  $\omega_0\in L^2$.  Then, the following hold:
   \begin{enumerate}
   \item There exists a weak solution to \eqref{bouss_aniso} ($P_{\nu_x,0}^0$) in the sense of Definition \ref{def:soln_aniso}.
  \item  If $\omega_0\in L^p$,  and $\theta_0\in L^p$,  with  $p\in [2,\infty)$ fixed, then this weak solution satisfies $\omega\in L^\infty([0,T],L^p)$ and  $\theta\in L^\infty([0,T],L^p)$.
  \item Furthermore, if  $\omega_0\in \sqrt{L}$ and $\theta_0\in L^\infty$, then there exists a solution $\omega\in L^\infty([0,T], \sqrt{L})\cap C_w([0,T], L^2)$, $\pd{\bu}{t}\in L^2([0,T], V')$ and $\theta\in L^\infty([0,T],L^\infty)\cap C([0,T],w\mbox{\tac-} L^\infty)$ (where $w\mbox{\tac-} L^\infty$ denotes the weak-$*$ topology on $L^\infty$) with $\pd{\theta}{t}\in L^\infty([0,T], H^{-1})$.

  \end{enumerate}
\end{theorem}

\begin{proof}   The outline of our proof is as follows.  We begin by generating approximate sequence of solutions $(\bun,\thetan)$ to $P^0_{\nu_x,0}$ by adding artificial {\em vertical viscosity} $\nu_y^{(n)}>0$,  artificial diffusion $\kappa^{(n)}>0$, where $\kappa^{(n)}, \nu_y^{(n)}\maps0$ as $n\maps \infty$, and also by smoothing the initial data.  Global existence of solutions to the {\em fully viscous} system $P^0_{\nu,\kappa}$, given smoothed initial condition is guaranteed (see, Theorem \ref{thm:diffusion} part (iii)).  Next, we establish uniform bounds, for the relevant norms of the approximate sequence of solutions which are independent of $n$ using basic energy estimates.  We then employ the Aubin Compactness Theorem (see, e.g., \cite{Temam_2001_Th_Num, Constantin_Foias_1988}) to show that the sequence of approximate solutions has a subsequence converging in appropriate function spaces.  This limit will serve as a candidate weak solution.  We then show that one can pass to the limit  to show that the candidate functions satisfy the weak formulation \eqref{bouss_aniso_weak_form}.  Then we establish some regularity results.

\begin{list}{}{\leftmargin=0em}
\item {\em {\bf Step 1: } Generating solutions to the regularized system given smoothed initial data.}

Let $\nu_x > 0$ be fixed and let $\kappa^{(n)}, {\nu^{(n)}_y}$ be a sequence of positive numbers, converging to zero.  In fact, we can also assume that both $\kappa^{(n)}\leq \nu_x$ and ${\nu^{(n)}_y}\leq\nu_x$.
Let $(\bun_0, \thetan_0)$ is a sequence of smooth initial data such that $\bun_0\maps\bu_0$ in $V$ and $\thetan_0\maps\theta_0$ in $L^2$, chosen in such a way that for each $n\in\nN$, $\|\bun_0\|\leq \|\bu_0\|+ \frac{1}{n}$ and $|\thetan_0|\leq |\theta_0|+\frac{1}{n}$.    Notice, since $\bun_0$ is smooth it follows that $\nabla^\perp\cdot \bun_0 = \omegan_0$ and so $\omegan_0$ are smooth functions bounded in $L^2$.  From Theorem \ref{thm:diffusion} part (iii),  by slightly modifying the proof of this result to account for values of the viscosity which differ in the horizontal and vertical directions, we have that for each $n$, there exist $(\bun, \thetan)$ satisfying the following equations:

\begin{subequations}\label{bouss_aniso+nu_y_weak_form}
\begin{align}
   \quad
   -\int_0^T(\bun(s),\Phi'(s))\,ds&+\nu_x\int_0^T(\partial_1\bun(s),\partial_1\Phi(s))\,ds
   \\\notag
+ \nu^{(n)}_y&\int_0^T(\partial_2\bun(s),\partial_2\Phi(s))\,ds
   +\sum_{j=1}^2\int_0^T(u^{j,(n)}\bun,\partial_j\Phi)\,ds
   \\&\notag
    = (\bun_0,\Phi(0))+\int_0^T(\thetan(s)\be_2,\Phi(s))\,ds
   \\\notag\\
   -\int_0^T(\thetan(s),\vphi'(s))\,ds &+    \kappa^{(n)}\int_0^T(\nabla\thetan(s),\nabla\Phi(s))\,ds           +\int_0^T(\thetan\bun,\nabla\vphi)\,ds \\&= (\thetan_0,\vphi(0)).
\end{align}
\end{subequations}

\item {\em {\bf Step 2: } A priori estimates  and using compactness arguments to prove convergence of a subsequence.}

We next establish {\em a priori} estimates on $(\bun,\thetan)$ uniformly in $n$ (independent of $\nu_y^{(n)}$ and $\kappa^{(n)}$).  From the above smoothness properties of $(\bun,\thetan)$, we can now derive {\em a priori} estimates using basic energy estimates in which the derivatives and integrations are well defined.  First, one can obtain, because div $\bun$=0, that
\begin{align}\label{eq:thetan_l2_est}
   |\thetan(t)|\leq|\thetan_0|\leq |\theta_0| + \frac{1}{n},
\end{align}
and
\begin{align*}
|\bun(t)|^2+2\nu_x\int_0^t|\partial_1\bun(\tau)|^2\,d\tau &+ 2\nu_y^{(n)}\int_0^t|\partial_2\bun(\tau)|^2\,d\tau\\
&\leq (|\bu_0|+ \frac{1}{n}+t(|\theta_0|+ \frac{1}{n}))^2.
\end{align*}
The calculations above are justified by replacing the test functions by $\thetan$ and $\bun$ in \eqref{bouss_aniso+nu_y_weak_form} and then integrating by parts.

Using the evolution equation of the vorticity, namely the equation
\begin{align}\label{eq:omega_n}
   \partial_t \omegan+\bun\cdot\nabla\omegan-\nu_x\partial_1^2\omegan-\nu_y^{(n)}\partial_2^2\omegan =\thetan_x,
\end{align}
we also have
\begin{align*}
   \frac{1}{2}\frac{d}{dt}|\omegan|^2+\nu_x|\partial_1^2\omegan| + \nu_y^{(n)}|\partial_2^2\omegan|
   &=
   -(\thetan,\partial_1\omegan)
   \\&\quad
   \leq\frac{\nu_x}{2}|\partial_1\omegan|^2+\frac{1}{2\nu_x}|\thetan|^2.
\end{align*}
Integrating this gives
\begin{align}\label{omega_LiL2_omega1_L2H1}
   |\omegan|^2+\nu_x\int_0^t|\partial_1\omegan|^2\,d\tau &+ 2\nu_y^{(n)}\int_0^t|\partial_2\omegan|^2\,d\tau\\&\leq \left(|\omega_0|+ \frac{1}{n}\right)^2+\frac{t}{2\nu_x}\left(|\theta_0|+ \frac{1}{n}\right)^2,
\end{align}
which implies that $\omegan$ is uniformly bounded in $L^\infty([0,T],L^2)$ with respect to $n$, and therefore $\bun$ is uniformly bounded in $L^\infty([0,T],V)$ with respect to $n$.  Furthermore, \eqref{omega_LiL2_omega1_L2H1} shows that $\partial_1\omegan$ is uniformly bounded in $L^2([0,T],L^2)$ with respect to $n$.  We also observe that
\begin{align*}
   \partial_1\omegan
   &=
   \partial_1^2 u^{2,(n)}- \partial_1\partial_2 u^{1,(n)}
   =\partial_1^2u^{2,(n)}+\partial_2^2u^{2,(n)} = \triangle u^{2,(n)}.
\end{align*}
Therefore, $\triangle u^{2,(n)}$ is uniformly bounded in $L^2([0,T],L^2)$, so that $u^{2,(n)}$ is uniformly bounded in $L^2([0,T],H^2)$ by elliptic regularity, and thus $\nabla u^{2,(n)}$ is uniformly bounded in $ L^2([0,T],H^1)$, all with respect to $n$.
Next we derive uniform bounds on the derivatives $(\pd{\bun}{t})_{n\in \nN}$.  Note that
%
%
$$\pd{\omegan}{t} = -\mathcal{B}(\omegan,\bun) + \nu_x\partial_1^2\omegan + \nu_y^{(n)} \partial_2^2\omegan + \partial_1\thetan$$
Thus,
\begin{equation}\label{eq:w_t_H2}
\aligned
   \norm{\pd{\omegan}{t}}_{H^{-2}}
   &\leq
   \sup_{\|\bw\|_{\dot{H}^2}=1}\abs{\ip{\mathcal{B}(\omegan,\bun)}{\bw}} + \nu_x \sup_{\|\bw\|_{\dot{H}^2}=1}\abs{\ip{  \partial_1^2\omegan   }{\bw}}   \\
   &\quad+  \nu_y^{(n)} \sup_{\|\bw\|_{\dot{H}^2}=1}\abs{\ip{ \partial_2^2\omegan   }{\bw}}  +\sup_{\|\bw\|_{\dot{H}^2}=1}\abs{\ip{   \partial_1\thetan   }{\bw}}   \\
   &=
    \sup_{\|\bw\|_{\dot{H}^2}=1}\abs{\ip{\omegan\bun}{\nabla\bw}} + \nu_x \sup_{\|\bw\|_{\dot{H}^2}=1}\abs{\ip{  \omegan   }{\partial_1^2\bw}}   \\
   &\quad+  \nu_y^{(n)} \sup_{\|\bw\|_{\dot{H}^2}=1}\abs{\ip{ \omegan   }{  \partial_2^2\bw}}  +\sup_{\|\bw\|_{\dot{H}^2}=1}\abs{\ip{  \thetan   }{\partial_1\bw}}   \\
         &\leq
    |\omegan||\bun|^{1/2}\|\bun\|^{1/2} + \nu_x|\omegan| + \nu_x |\omegan| + |\thetan| ,
\endaligned
\end{equation}
Since each of the terms on the right-hand side of the inequality above is bounded independently of $n$, we deduce by the Calder\'on-Zygmund elliptic estimate \eqref{Calderon} that $\partial_t\bun$ is bounded in $L^\infty([0,T],V')$ independently of $n$.  Similarly, one can show easily that
\begin{equation}
\norm{\pd{\thetan}{t}}_{H^{-2}} \leq |\thetan||\bun|^{1/2}\|\bun\|^{1/2},
\end{equation}
which implies also that $\pd{\thetan}{t}$ is bounded in  $L^\infty([0,T],H^{-2})$ independently of $n$.  To summarize, we have from the above results that
\begin{subequations}\label{aniso_bounds}
\begin{align}
 (\thetan)_{n\in\nN}\quad &\mbox{ is bounded in } \quad L^\infty([0,T], L^2), \\
(\bun)_{n\in\nN} \quad&\mbox{ is bounded in } \quad L^\infty([0,T], V), \\
(u^{2,(n)})_{n\in\nN}\quad &\mbox{ is bounded in }\quad L^2([0,T], H^2)\\
\left(\pd{\bun}{t}\right)_{n\in\nN} \quad&\mbox{ is bounded in } \quad L^\infty([0,T], V'), \\
\left(\pd{\thetan}{t}\right)_{n\in\nN} \quad&\mbox{ is bounded in } \quad L^\infty([0,T], H^{-2}).
\end{align}
\end{subequations}
 Using  Banach-Alaoglu and Aubin Compactness theorems (see, e.g., \cite{Temam_2001_Th_Num, Constantin_Foias_1988}), the uniform bounds with respect to $n$ as stated in \eqref{aniso_bounds} implies that one can extract a further subsequence (which we relabel with the index $n$ if necessary) such that

\begin{subequations}\label{wk_conv_aniso}
\begin{align}
\label{wk_theta_LiH_aniso}
\thetan\rightharpoonup\theta &\quad\text{weakly in }L^2([0,T],L^2)\text{ and weak-$*$ in }L^\infty([0,T],L^2).\\
\label{st_u_L2H_aniso}
\bun\maps\bu &\quad\text{strongly in }L^2([0,T],H),\\
\label{wk_u_L2V_aniso}
\bun\rightharpoonup\bu &\quad\text{weakly in }L^2([0,T],V)\text{ and weak-$*$ in }L^\infty([0,T],V),\\
\label{wk_u_squared_aniso}
u^{2,(n)}\rightharpoonup \quad& u^{2,(n)} \quad\text{weakly in }L^2([0,T],H^2),\\
\label{wk_du_dt_aniso}
\pd{\bun}{t}\rightharpoonup\pd{\bu}{t} &\quad\text{weakly in }L^2([0,T],V')\text{ and weak-$*$ in }L^\infty([0,T],V'),\\
\label{wk_dtheta_dt_aniso}
\pd{\thetan}{t}\rightharpoonup\pd{\theta}{t} &\quad\text{weakly in }L^2([0,T],H^{-2})\text{ and weak-$*$ in }L^\infty([0,T],H^{-2}).
\end{align}
\end{subequations}

\item {\em {\bf Step 3: } Pass to the limit in the system.}

It remains to show that  \eqref{wk_conv_aniso} is enough to pass to the limit in \eqref{bouss_aniso+nu_y_weak_form} to show that $(\bu,\theta)$ satisfies \eqref{bouss_aniso_weak_form}.    To do this, in accordance with Remark~\ref{test_fcns_are_trig_polys} and Definition~ \ref{def:soln_aniso}, we only consider test functions of the form \eqref{test_fcns}, which we note is sufficient for showing that $(\bu,\theta)$ satisfies \eqref{bouss_aniso_weak_form}. For the linear terms in \eqref{bouss_aniso+nu_y_weak_form}, we have, by the weak convergence in \eqref{wk_u_L2V_aniso}  and  \eqref{wk_theta_LiH_aniso}, as $n\maps \infty$ (that is, $\kappa^{(n)} , \nu_y^{(n)}\maps 0$),
\begin{align*}
\int_0^T(\bun(s), \Gamma'_{\bm}(s) e^{2\pi i\bm\cdot \bx})\,ds
  &\maps
  \int_0^T(\bu(s), \Gamma'_{\bm}(s) e^{2\pi i\bm\cdot \bx})\,ds
   ,\\
\nu_x\int_0^T(\partial_1\bun(s), \Gamma_{\bm}(s) \partial_1e^{2\pi i\bm\cdot \bx}      )\,ds
&\maps
\nu_x\int_0^T(\partial_1\bu(s), \Gamma_{\bm}(s) \partial_1e^{2\pi i\bm\cdot \bx}      )\,ds
,\\
\int_0^T(\thetan(s)\be_2,\Gamma_{\bm}(s) e^{2\pi i\bm\cdot \bx} )\,ds
&\maps
\int_0^T(\theta(s)\be_2,\Gamma_{\bm}(s) e^{2\pi i\bm\cdot \bx} )\,ds
,\\
\int_0^T(\thetan(s),e^{2\pi i\bm\cdot\bx})\chi_{\bm}'(s)\,ds
&\maps
\int_0^T(\theta(s),e^{2\pi i\bm\cdot\bx})\chi_{\bm}'(s)\,ds
,\\
 \kappa^{(n)}\int_0^T(\partial_2\bun(s), \Gamma_{\bm}(s) \partial_2e^{2\pi i\bm\cdot \bx}    )\,ds &  \maps
0,\\
\kappa^{(n)}\abs{\int_0^T((\theta_\kappa(s), e^{2\pi i\bm\cdot \bx}\chi_{\bm}(s)))\;ds}
& \leq C\sqrt{\kappa^{(n)} }\pnt{\sqrt{\kappa^{(n)}}\|\theta_\kappa\|_{L^2_TH^1_x}}
\\&\leq CK_0\sqrt{\kappa^{(n)}}
\maps
0.
\end{align*}

\bigskip


It remains to show the convergence of the remaining non-linear terms.  Let
\begin{align*}
   I(n)&:=\sum_{j=1}^2\int_0^T(u^{j,(n)}\bun, \Gamma_{\bm}(s)\partial_j e^{2\pi i\bm\cdot \bx} )\,ds-
   \sum_{j=1}^2\int_0^T(u^j\bu, \Gamma_{\bm}(s) \partial_j e^{2\pi i\bm\cdot \bx} )\,ds
   \\
   J(n)&:=\int_0^T\adv{\bun(s)\thetan(s),\chi_{\bm}(s)\nabla e^{2\pi i\bm\cdot \bx}}\,ds-\int_0^T\adv{\bu(s)\theta(s),\chi_{\bm}(s)\nabla e^{2\pi i\bm\cdot \bx}}\,ds.
\end{align*}

To show $I(n)\maps0$ as $n\maps\infty$, we write $I(n)=I_1(n)+I_2(n)$, the definitions of which are given below.
We  have
\begin{align*}
   |I_1(n)|
   &:=
   \abs{\sum_{j=1}^2\int_0^T((u^{j,(n)}(s)-u^j(s) ) \bun(s),\partial_j e^{2\pi i\bm\cdot \bx})\chi_{\bm}(s)\,ds
   }\\
    &\leq
    \int_0^T |\bun(s)-\bu(s) ||\bun(s)||\nabla e^{2\pi i\bm\cdot \bx}\chi_{\bm}(s)|\,ds\\
    &\leq\|\bun-\bu\|_{L^2_TH_x}\|\bun\|_{L^\infty_TH_x}\|\nabla e^{2\pi i\bm\cdot \bx}\chi_{\bm}\|_{L^2_TL^\infty_x}\maps 0,
    \end{align*}
as $n\maps\infty$, since $\bun\maps\bu$ strongly in $L^2([0,T],H)$ and $\bun$ is uniformly bounded in $L^\infty([0,T],V)$ and hence in $L^\infty([0,T],H)$ .  Similarly, for $I_2$, we have that as $n\maps \infty$
\begin{align*}
  I_2(n)&:=
  \sum_{j=1}^2\int_0^T\left(u^j(s)(\bun(s)-\bu(s)),\partial_j e^{2\pi i\bm\cdot \bx}\right)\chi_{\bm}(s)\,ds
     \maps0.
\end{align*}

To show $J(n)\maps0$ as $n\maps\infty$, we write $J(n)=J_1(n)+J_2(n)$.  We  have
\begin{align*}
   J_1(n)
   &:=
   \int_0^T((\bun(s)-\bu(s) ) \thetan(s),\nabla e^{2\pi i\bm\cdot \bx})\chi_{\bm}(s)\,ds
     \maps0,
\end{align*}
as $n\maps\infty$, since $\bun\maps\bu$ strongly in $L^2([0,T],H)$ and $\thetan\maps\theta$ weakly in $L^2([0,T],H)$.  For $J_2$, we have
\begin{align*}
   J_2(n)&:= \int_0^T\left(\bu(s)(\thetan(s)-\theta(s)),\nabla e^{2\pi i\bm\cdot \bx}\right)\chi_{\bm}(s)\,ds
     \maps0,
\end{align*}
by the weak convergence in \eqref{wk_theta_LiH_aniso} and the fact that $\bu\in L^2([0,T],H)$.  This establishes the existence of weak solution to the system $P^0_{\nu_x,0}$ when $\bu_0\in H^1 \mbox{ and } \theta_0\in L^2$.\\

\item {\em {\bf Step 4: } Show that $\omega\in C_w([0,T];L^2)$.}

By  the Arzela-Ascoli theorem, it suffices to show that (a) $\{\omegan\}$ is a relatively weakly compact set in $L^2(\nT^2)$
for a.e $t\geq 0$ and (b) for every $\phi\in L^2(\nT^2)$ the sequence $\{(\omegan,\phi)\}$ is equicontinuous in $C([0,T])$. Condition (a) follows from the uniform boundedness of $\omegan$  in $L^2(\nT^2)$ for a.e. $t\geq 0$ given in \eqref{omega_LiL2_omega1_L2H1}.  Next, we show that condition (b) is satisfied.
We follow similar argument as in  Step 3 of Section \ref{s:P_nu_zero_zero} equation \eqref{e:continuity_theta_kappa}, where, we start by assuming that $\phi$ is a trigonometric polynomial to obtain,
\begin{align*}
&\quad
|(\omegan(t_2),\phi) - (\omegan(t_1), \phi)|\\
&\leq |\nu_x\int_{t_1}^{t_2}(\partial_1\omegan(t),\partial_1\phi)\;dt |+   |\nu_y\int_{t_1}^{t_2}(\partial_2\omegan(t),\partial_2\phi)\;dt |\\
&\quad +  |\int_{t_1}^{t_2} (\bun\cdot\nabla\phi,\omegan)\;dt|+ |\int_{t_1}^{t_2}(\thetan,\partial_x\phi)\;dt| \\
&\leq\nu_x \int_{t_1}^{t_2}  |\partial_1\omegan||\partial_1\phi| \;dt + \nu_x \int_{t_1}^{t_2}  |\omegan| |\partial_2^2\phi | \;dt\\
&\quad + \|\nabla\phi\|_\infty\int_{t_1}^{t_2}  |\bun||\omegan| \;dt +  \int_{t_1}^{t_2}  |\thetan||\nabla\phi| \;dt\\
&\leq |\nabla\phi|_\infty| |t_2-t_1|^{1/2} \nu_x\int_{t_1}^{t_2}|\partial_1\omegan|^2\;dt + |\partial_2^2\phi|_\infty||t_2-t_1|\|\omegan\|_{L^\infty_TL^2_x}\\
&\quad+ \|\nabla\phi\|_\infty |t_2-t_1|\left(\|\bun\|_{L^\infty_TL^2_x} \|\omegan\|_{L^\infty_TL^2_x}+ \|\thetan\|_{L^\infty_TL^2_x}\right),
\end{align*}
where recall we have assumed without loss of generality that $\nu_y^{(n)}<\nu_x$.
From the uniform boundedness of $\omegan$ \eqref{omega_LiL2_omega1_L2H1} and $\thetan$ \eqref{eq:thetan_l2_est}, the right-hand side can be made small when $|t_2-t_1|$ is small enough.  Thus we have that the set $\{(\omegan,\phi) \}$ is equicontinuous in $C([0,T])$.  Then one can extend this result for all test functions $\phi$ in $L^2(\nT^2)$ using a simple density argument as before.  This completes the proof of part (1) of Theorem \ref{EU_aniso}.\\

\item {\em {\bf Step 5}  Proof of part (2) of Theorem \ref{EU_aniso}.}

We choose a sequence of smooth initial data $\omegan_0\maps \omega_0$ and similarly $\thetan_0\maps \theta_0$ in every $L^p$ with $p\geq 2$ chosen in such a way that for each $n\in\nN$, $\|\omegan_0\|_p\leq \|\omega_0\|_p+ \frac{1}{n}$ and $\|\thetan_0\|_p\leq\|\theta_0\|_p+\frac{1}{n}$.   From Theorem \ref{thm:diffusion}, we obtain for each $n$, a solution $u^{(n)}\in H^3$ which then gives us $\omegan\in H^2$  which is a topological algebra, hence $|\omegan|^{p-2}\omegan \in H^{2}$. We take the inner product of  \eqref{eq:omega_n} with $|\omegan|^{p-2}\omegan$. Integrating by parts, we have
\begin{align*}
     \frac{1}{p}\frac{d}{dt}\|\omegan\|_p^p & +\nu_x(p-1)\int_{\nT^2}|\partial_1\omegan|^2|\omegan|^{p-2}\,dx + \nu_y^{(n)}(p-1)\int_{\nT^2}|\partial_2\omegan|^2|\omega|^{p-2}\,dx
   \\&\leq
   (p-1)\int_{\nT^2}|\thetan||\partial_1\omegan||\omegan|^{p-2}\,dx
   \\&\leq
   \nu_x(p-1)\int_{\nT^2}|\partial_1\omegan|^2|\omegan|^{p-2}\,dx
   +\frac{p-1}{4\nu_x}\int_{\nT^2}|\thetan|^2|\omegan|^{p-2}\,dx
   \\&\leq
   \nu_x(p-1)\int_{\nT^2}|\partial_1\omegan|^2|\omegan|^{p-2}\,dx
   +\frac{p-1}{4\nu_x}\|\thetan\|_p^2\|\omegan\|_p^{p-2}.
\end{align*}
Therefore, we have
\begin{align*}
   \frac{1}{p}\frac{d}{dt}\|\omegan\|_p^p
   \leq
   \frac{p-1}{4\nu_x}\|\thetan\|_p^2\|\omegan\|_p^{p-2}
   \leq
   \frac{p-1}{4\nu_x}\left(\|\theta_0\|_p+\frac{1}{n}\right)^2\|\omegan\|_p^{p-2}.
\end{align*}
That is,
\begin{align*}
   \frac{d}{dt}\|\omegan\|_p^2
   \leq
   \frac{p-1}{2\nu_x}\|\thetan_0\|_p^2
   \leq
   \frac{p-1}{2\nu_x}\left(\|\theta_0\|_p+\frac{1}{n}\right)^2.
\end{align*}
Integrating in time, we have
\begin{align}\label{omegan_p_bounds}
   \|\omegan(t)\|_p^2
   &\leq
   \|\omegan_0\|_p^2+\frac{p-1}{2\nu_x}\left(\|\theta_0\|_p+\frac{1}{n}\right)^2t\\
    &\leq\notag
   \left(\|\omega_0\|_p+\frac{1}{n}\right)^2+\frac{p-1}{2\nu_x}\left(\|\theta_0\|_p+\frac{1}{n}\right)^2t.
\end{align}
That is,   $\omegan$ is uniformly bounded in $L^\infty([0,T], L^p)$  for each $p \in [2, \infty)$, independent of $n$. It follows from the Banach-Alaoglu Theorem and diagonalization process,  that there exists a further subsequence which we also denote as $\omegan$ converging weak-\tac \, in  $L^\infty([0,T], L^p)$ to some limit  which we denote as $\omega$ and this  limit also enjoys the limit of the upper bound, that is
\begin{align}\label{omega_p_bound}
 \|\omega\|_p^2\leq  \left(\|\omega_0\|_p+\frac{1}{n}\right)^2+\frac{p-1}{2\nu_x}\left(\|\theta_0\|_p+\frac{1}{n}\right)^2t.
 \end{align}
 This implies that $\omega\in L^\infty([0,T];L^p)$ for all $p \in [2, \infty)$. Similarly we find that
 \begin{align}\label{thetan_lp_bound}
 \|\thetan(t)\|_p\leq\|\thetan_0\|_p\leq\|\theta_0\|_p + \frac{1}{n},
 \end{align}
 which implies that $\thetan$ converges weak-\tac \, in $L^\infty([0,T];L^p)$ to $\theta \in L^\infty([0,T];L^p)$ for all $p \in [2, \infty)$, and $\|\theta\|_{L^\infty([0,T], L^p)} \leq \|\theta_0\|_p$.\\

\item {\em {\bf Step 6}  Proof of part (3) of Theorem \ref{EU_aniso}.}

To prove Theorem \ref{EU_aniso} {\em part (3) } we divide both sides of  \eqref{omega_p_bound} by $p-1$ and then taking the supremum over all $p>2$ of both sides, we get that $\omega\in L^\infty([0,T], \sqrt{L})$ provided that $\omega_0\in \sqrt{L}$ and $\theta_0\in L^\infty$. Next, we want to show that $\theta\in C([0,T];w\mbox{\tac-} L^\infty)$.  We will use the Arzela-Ascoli theorem as in Step 4.   Notice that if $\theta_0\in L^\infty$ then \eqref{thetan_lp_bound} holds uniformly for all $p\in[2,\infty)$ and hence
\begin{align}\label{thetan_infty_bound}
 \|\thetan(t)\|_\infty\leq\|\theta_0\|_\infty + \frac{1}{n}.
 \end{align}
This implies that the sequence $\thetan(t)$ is a relatively compact set in the weak$-*$ topology of $L^\infty([0,T]\times\nT^2)$.   It suffices to show that the sequence $\{\left(\thetan,\phi\right)\}$ is equicontinuous in $C([0,T])$ for every $\phi\in L^1$.  It follows automatically from the previous result and the density of $L^2(\nT^2)$ in $L^1(\nT^2)$ that $\theta\in C_w([0,T],L^2) $.
Finally, we would like to show that $\pd{\theta}{t}\in L^\infty([0,T], H^{-1})$ and hence $\pd{\theta}{t}\in L^2([0,T], H^{-1})$.  Since $\omega\in L^\infty([0,T],\sqrt{L})$, we have in particular that $\omega\in L^\infty([0,T],L^3)$, and hence $\bu\in L^\infty([0,T],W^{1,3})\subset L^\infty([0,T],L^\infty)$ by \eqref{Calderon}, \eqref{poincare}, and the Sobolev Embedding Theorem.  From equation \eqref{e:funct_2_horizontal}, using \eqref{B_theta:def} and the fact that $\theta\in L^\infty([0,T],L^2)$,  we obtain,
\begin{align} \label{e:strong_aniso_D_theta}
\norm{\pd{\theta}{t}}_{H^{-1}} = \sup_{\|w\| = 1} \abs{\ip{\mathcal{B}(\bu,\theta)}{w}} \leq\|\bu\|_\infty|\theta| <\infty \text{ a.e } t\in [0,T].
\end{align}
This completes the proof of part (3) of Theorem \ref{EU_aniso}.
\qedhere
\end{list}
\end{proof}



\begin{theorem}[Uniqueness for the Anisotropic Case]
   Let $\theta_0\in L^\infty$, $\omega_0\in \sqrt{L}$.  Then, for every $T>0$, there exists a unique  solution  $\omega\in L^\infty([0,T], \sqrt{L})\cap C_w([0,T];L^2)$ and $\theta\in L^\infty([0,T], L^\infty)\cap C([0,T]),w\mbox{\tac-} L^\infty)$  to \eqref{bouss_aniso}.
\end{theorem}

\begin{proof}
Let $T>0$ arbitrarily large.  The existence of solution on the interval $[0,T]$ is established above, therefore it suffices to show uniqueness.  We note that some very important {\em a priori} estimates that we need in the beginning of this proof were first elegantly derived in \cite{Danchin_Paicu_2008}.   We recall those estimates that we have borrowed from \cite{Danchin_Paicu_2008}.  We have derived them
rigorously in the previous theorem and we derive them here again formally to make the proof of uniqueness self-contained.  First, one may easily show that for any $p\in[2,\infty]$, we have
\begin{align}\label{eq:theta_lp_est}
   \|\theta(t)\|_{p}\leq\|\theta_0\|_{p},
\end{align}
so $\theta\in L^\infty([0,T],L^p)$, $p\in[2,\infty]$.
Given that $\omega_0\in \sqrt{L}$, and hence $\omega_0\in L^2$, we have
\begin{align*}
   \frac{1}{2}\frac{d}{dt}|\omega|^2+\nu|\partial_1^2\omega|= -(\theta,\partial_1\omega)
   \leq\frac{\nu}{2}|\partial_1\omega|^2+\frac{1}{2\nu}|\theta|^2.
\end{align*}
Integrating this gives
\begin{align*}
   |\omega|^2+\nu\int_0^t|\partial_1\omega|^2\,d\tau\leq |\omega_0|^2+\frac{t}{\nu}|\theta_0|^2.
\end{align*}
This implies that $\omega\in L^\infty([0,T],L^2)$, and therefore $\bu\in L^\infty([0,T],V)$.  Furthermore, $\partial_1\omega\in L^2([0,T],L^2)$.  Using the divergence free condition \eqref{bouss_aniso_div}, we observe that
\begin{align*}
   \partial_1\omega
   &=
   \partial_1^2 u^2- \partial_1\partial_2 u^1
   =\partial_1^2u^2+\partial_2^2u^2 = \triangle u^2.
\end{align*}
Therefore, $\triangle u^2\in L^2([0,T],L^2)$, so that $u^2\in L^2([0,T],H^2)$ by elliptic regularity, and thus $\nabla u^2\in L^2([0,T],H^1)$.  By inequality \eqref{CZ_est}, we have
\begin{align}\label{Du2_root_L}
   \|\nabla u^2\|_{p}\leq C\sqrt{p-1}\|\nabla u^2\|_{H^1}.
\end{align}
so that $\nabla u^2\in L^2([0,T],\sqrt{L})$.

Next, we recall that we have global in time control over the $\|\omega\|_{\sqrt{L}}$.   Taking the inner product of \eqref{bouss_aniso_vort} with $|\omega|^{p-2}\omega$ for some $p>2$ and integrating by parts, and integrating in time, we have
\begin{align}
   \|\omega(t)\|_{p}^2
   \leq
   \|\omega_0\|_{p}^2+\frac{p-1}{2\nu}\|\theta_0\|_{p}^2t.
\end{align}
This shows that $\omega\in L^\infty([0,T],\sqrt{L})$. Using this, and the facts that  $\partial_1u^1=-\partial_2u^2$ (by \eqref{bouss_aniso_div}) and $\partial_2 u^1=\partial_1 u^2-\omega$, we have thanks to \eqref{Du2_root_L} that $\nabla u^1\in L^2([0,T],\sqrt{L})$.  Combining this with \eqref{Du2_root_L} shows that
\begin{align}\label{grad_u_sqrt_L}
   \nabla\bu\in L^2([0,T],\sqrt{L}).
\end{align}
We recall again that all the estimates above were first derived in \cite{Danchin_Paicu_2008} for the case where $\Omega = \nR^2$.



We are now ready to show that if $(\bu_1,\theta_1)$ and $(\bu_2,\theta_2)$ are two solutions to \eqref{bouss_aniso_weak_form} on the interval $[0,T]$,  with the same initial data $(\bu_0,\theta_0)$ then they must be the equal.  Define $\bud := \bu_1-\bu_2$, $\diff{\theta}:=\theta_1-\theta_2$, and $\xi_\ell:=\triangle^{-1}\theta_\ell$, $\ell=1,2$, and $\xid:=\xi_1-\xi_2$.   Based on Remark \ref{e:functional_nu_horizontal}, these quantities satisfy the following functional equations.
\begin{subequations}\label{e:functional_nu_horizontal_diff}
  \begin{align}\label{e:funct_1_horizontal_diff}
  \pd{\bud}{t} + \nu \partial_{11}\bud + B(\bud,\bu_1)  + B(\bu_2,\bud) &= P_\sigma(\triangle\xid\be_2)\quad \mbox{in}\quad L^2([0,T], V')\quad
  \mbox{and}\\\label{e:funct_2_horizontal_diff}
   \pd{\triangle\xid}{t} + \mathcal{B}(\bud,\triangle\xi_1) + \mathcal{B}(\bu_2,\triangle\xid) &= 0\quad \mbox{in}\quad L^2([0,T], H^{-1}).
   \end{align}
   \end{subequations}
 Taking the action of \eqref{e:funct_1_horizontal_diff} on $\bud$ in $L^2([0,T], V)$ and the action of  \eqref{e:funct_2_horizontal_diff} in $L^2([0,T],H^{-1})$ with $\xid\in L^2([0,T], H^2)$, thanks to the properties of the operator $B$ in Lemma \ref{B:prop} and  the operator $\mathcal{B}$ in Lemma \ref{B_theta:prop}  we obtain the following:
\begin{align*}
   \frac{1}{2}\frac{d}{dt}|\bud(t)|^2+\nu\|\partial_1\bud\|^2 &= \sum_{j=1}^2(\diff{u}^j\bu_1,\partial_j\bud)    + (\triangle\xid\be_2,\bud)\\
   \frac{1}{2}\frac{d}{dt}\|\xid(t)\|^2&=-(\bud\triangle\xi_1,\nabla\xid)-(\bu_2\triangle\xid,\nabla\xid),
\end{align*}
where again we have used Lions-Magenes Lemma (see, e.g., \cite{Temam_2001_Th_Num}) to get that $ \ip{\pd{\bud}{t}}{\bud}=\frac{1}{2}\frac{d}{dt}|\bud(t)|^2$ and $ \ip{\pd{\triangle\xid}{t}}{\xid}=\frac{1}{2}\frac{d}{dt}\|\xid(t)\|^2$.
By Lemma \ref{B:prop}, we obtain
\begin{align*}
   \frac{1}{2}\frac{d}{dt}|\bud|^2+\nu|\partial_1\bud|^2
   &\leq  \int_{\nT^2}\abs{\nabla\bu_1}\abs{\bud}^2\,d\bx
  +\abs{(\triangle\xid\be_2,\bud)}
   \intertext{and}
   \frac{1}{2}\frac{d}{dt}\|\xid\|^2
   &\leq
   \abs{\int_{\nT^2}\bud\cdot\nabla\xid\triangle\xi_1\,d\bx}+\abs{\int_{\nT^2}\bu_2\cdot\nabla\xid\triangle\xid\,d\bx}.
\end{align*}
Next, observe that, due to the divergence free condition, $\be_1\cdot\partial_1\bud=-\be_2\cdot\partial_2\bud$, we have
\begin{align*}
   |(\triangle\xid\be_2,\bud)|
   &\leq
   \int_{\nT^2}\pnt{|\partial_1\xid\be_2\cdot\partial_1\bud|
   +|\partial_2\xid\be_2\cdot\partial_2\bud|}\,d\bx
    \\&=
    \int_{\nT^2}\pnt{|\partial_1\xid\be_2\cdot\partial_1\bud|
   +|\partial_2\xid\be_1\cdot\partial_1\bud|}\,d\bx
    \\&\leq
   \frac{1}{\nu}|\partial_1\xid|^2+\frac{\nu}{4}|\be_2\cdot\partial_1\bud|^2
   +\frac{1}{\nu}|\partial_2\xid|^2+\frac{\nu}{4}|\be_1\cdot\partial_1\bud|^2.
\end{align*}
%
Combining the above estimates, we find
\begin{align*}
   \frac{1}{2}\frac{d}{dt}|\bud|^2+\nu|\partial_1\bud|^2
   &\leq
  \int_{\nT^2}\abs{\nabla\bu_1}\abs{\bud}^2\,d\bx
  +\frac{2}{\nu}\|\xid\|^2 +\frac{\nu}{2}|\partial_1\bud|^2
  \\&\leq
  \|\bud\|_{\infty}^{2/p}\int_{\nT^2}\abs{\nabla\bu_1}\abs{\bud}^{2-2/p}\,d\bx
  +\frac{2}{\nu}\|\xid\|^2 +\frac{\nu}{2}|\partial_1\bud|^2
   \\&\leq
   \|\nabla\bu_1\|_{p}\|\bud\|_{\infty}^{2/p}|\bud|^{2-2/p}+\frac{2}{\nu}\|\xid\|^2 +\frac{\nu}{2}|\partial_1\bud|^2
\intertext{where we have used H\"older's inequality.  Similarly, by Lemma \ref{B_theta:prop}}
   \frac{1}{2}\frac{d}{dt}\|\xid\|^2
   &\leq
   \abs{\int_{\nT^2}\bud\cdot\nabla\xid\triangle\xi_1\,d\bx}+\int_{\nT^2}|\nabla\bu_2||\nabla\xid|^2\,d\bx
   \\&\leq
   |\bud||\nabla\xid|\|\triangle\xi_1\|_{\infty}
   +\|\nabla\bu_2\|_{p}\|\nabla\xid\|_{\infty}^{2/p}|\nabla\xid|^{2-2/p}.
\end{align*}
From the estimates above we can now adapt the well-known Yudovich argument for the 2D incompressible Euler equations (see, e.g., \cite{Yudovich_1963}) to complete the uniqueness proof.  Let $X^2:=|\bud(t)|^2+\|\xid(t)\|^2+\eta^2$ for some arbitrary $\eta>0$.  Adding the above two inequalities and using Young's inequality  gives,
\begin{align*}
   &\qquad
   \frac{1}{2}\frac{d}{dt}X^2+\frac{\nu}{2}|\partial_1\bud|^2
   \\&\leq
   K_{\nu}\pnt{|\bud|^2+\|\xid\|^2+\eta^2}
   \\&\quad
   +\pnt{\|\nabla\bu_2\|_{p}+\|\nabla\bu_1\|_{p}}%
   \pnt{\|\bud\|_{\infty}^{2/p}+\|\nabla\xid\|_{\infty}^{2/p}}
   \pnt{|\bud|^{2-2/p}+|\nabla\xid|^{2-2/p}}
   \\&\leq
   K_{\nu}X^2
   +C\pnt{\|\nabla\bu_2\|_{p}+\|\nabla\bu_1\|_{p}}%
   \pnt{\|\bud\|_{\infty}^{2/p}+\|\nabla\xid\|_{\infty}^{2/p}}
   X^{2-2/p}.
\end{align*}

Neglecting the term $\frac{\nu}{2}|\partial_1\bud|^2$, dividing by $X$, and making the change of variables $Y(t)=e^{-K_{\nu}t}X(t)$, we have after a simple calculation,
\begin{align*}
   \dot{Y} \leq
   Ce^{-2K_{\nu}t/p}\pnt{\|\nabla\bu_2\|_{p}+\|\nabla\bu_1\|_{p}}%
   \pnt{\|\bud\|_{\infty}^{2/p}+\|\nabla\xid\|_{\infty}^{2/p}}
   Y^{1-2/p}.
\end{align*}
Integrating this equation and using the fact that $e^{-2K_{\nu}t/p} \leq 1$, we get that
\begin{align*}
Y(t)\leq
  \left[ \eta^{2/p} +C \int_0^t \frac{1}{p}\pnt{\|\nabla\bu_2(s)\|_{p}+\|\nabla\bu_1(s)\|_{p}}%
   \pnt{\|\bud(s)\|_{\infty}^{2/p}+\|\nabla\xid(s)\|_{\infty}^{2/p}}\;ds \right]^{p/2}.
   \end{align*}
Letting $\eta\maps 0$ we discover that for all $t\in[0,T]$,
\begin{align}\label{p_eqn}
 \notag
 |\bud(t)|^2+\|\xid(t)\|^2
 &\leq  \pnt{\|\bud\|_{L_T^\infty L^\infty_x}+\|\nabla\xid\|_{L_T^\infty L^\infty_x}}
 \\&\qquad\cdot
\pnt{C\int_0^t\frac{1}{p}\pnt{\|\nabla\bu_2(s)\|_{p}+\|\nabla\bu_1(s)\|_{p}}\,ds}^{p/2}.
 \end{align}

Thanks to the fact that $\triangle\xid=\diff{\theta}\in L^\infty([0,T],L^\infty)\subset L^\infty([0,T],L^4)$, we have by elliptic regularity that $\xid\in L^\infty([0,T],W^{2,4})$, and therefore $\nabla\xid\in L^\infty([0,T],W^{1,4})$.   Thus, by the Sobolev Embedding Theorem, we have $\nabla\xid\in L^\infty([0,T],W^{1,4}) \subset L^\infty([0,T],C^{0,\gamma})$, for some $\gamma\in(0,1)$.  Furthermore, $\diff{\omega} \in L^\infty([0,T],\sqrt{L})$ implies, for instance that $\bud\in L^\infty([0,T],W^{1,4})$ by the Calder\'on-Zygmund elliptic estimate \eqref{Calderon}. Using the Sobolev Embedding Theorem again, we have $\bud\in L^\infty([0,T],C^{0,\gamma})$, for some $\gamma\in(0,1)$.  Therefore, the first factor on the right-hand side of \eqref{p_eqn} is bounded.
 Now, since $\nabla\bu_\ell\in L^2([0,T],\sqrt{L})$, $\ell=1,2$ by \eqref{grad_u_sqrt_L},  we have by Cauchy-Schwarz
\begin{align*}
\int_0^t\frac{\|\nabla\bu_\ell(s)\|_p}{p}\;ds
&\leq
\left(t\int_0^T\sup_{p\geq 2}\frac{\|\nabla\bu_\ell(s)\|^2_{p}}{p-1}\;ds \right)^{1/2}.
\end{align*}
Let $\displaystyle M_\ell = \int_0^T\sup_{p\geq 2}\frac{\|\nabla\bu_\ell(s)\|^2_{p}}{p-1}\;ds$, $\ell = 1, 2$ and $M = \max\{M_1, M_2\}$.  Thus, from the above, for every fixed $\tau\in (0,T]$ we have
\begin{align}\label{p_1}
|\bud(t)|^2+\|\xid(t)\|^2\leq K (2CM\tau)^{p/2}, \text{ for all } t\in [0,\tau],
\end{align}
 where the constant $C$ is the same constant which appears in \eqref{p_eqn} and $K= \pnt{\|\bud\|_{L_T^\infty L^\infty_x}+\|\nabla\xid\|_{L_T^\infty L^\infty_x}}$.  Now choose $\tau = \tau_0 = \min\{T, \frac{1}{4CM}\}$, and consider \eqref{p_1} on $[0,\tau_0]$.  Taking the limit as $p\maps \infty$, we get that
 $|\bud(t)|^2+\|\xid(t)\|^2\leq 0$ for all $t\in [0,\tau_0]$.  Restarting the time at $t = \tau_0$ and noting the fact that
\begin{align*}
\int_{\tau_0}^{t+\tau_0}\frac{\|\nabla\bu_\ell(s)\|_p}{p} \,ds
&\leq
\left(t\int_0^T\sup_{p\geq 2}\frac{\|\nabla\bu_\ell(s)\|^2_{p}}{p-1}\,ds \right)^{1/2},
\end{align*}
we obtain from the analogue of \eqref{p_eqn} on $[\tau_0,T]$ that $|\bud(t)|^2+\|\xid(t)\|^2\leq K(2CM\tau_0)^{p/2}$ for all $t\in [\tau_0,2\tau_0]$.  Since we defined $\tau_0 \leq \frac{1}{4CM}$, we take the limit $p\rightarrow \infty$ and find that on the interval $[\tau_0, 2\tau_0]$, we also have that $|\bud(t)|^2+\|\xid(t)\|^2\leq 0$.   We can continue this argument on the intervals $[2\tau_0, 3\tau_0], [3\tau_0, 4\tau_0],\dots,$ and so on.  Thus, we have  $|\bud(t)|^2+\|\xid(t)\|^2\leq 0$ for all $t\in [0,T]$.  This implies that, $|\bud(t)| = 0$ and $\|\xid(t)\| =0$ for all $t\in [0,T]$.
\end{proof}

\section{Global Well-posedness Results for the Voigt-regularized Inviscid and Non-diffusive Boussinesq Equations \texorpdfstring{($P_{0,0}^\alpha$)}{}} \label{s:P_alpha_zero_zero}

In this section, we investigate the problem $P_{0,0}^\alpha$, $\alpha>0$, given by \eqref{bouss_v} (with $\nu=\kappa=0$) in 2D.  We first establish global well-posedness results, and then investigate the behavior of solutions as $\alpha\maps 0$.  In particular, we compare the limiting behavior to sufficiently regular solutions of the $P_{0,0}^0$ problem.  This leads to a new criterion for the blow-up of solutions to the $P_{0,0}^0$ problem.  A similar criterion was given for the blow-up of the Surface Quasi-Geostrophic equations in \cite{Khouider_Titi_2008}, for the Euler equations in \cite{Larios_Titi_2009}, and for the inviscid, resistive MHD equations in \cite{Larios_Titi_2010_MHD}.

\begin{definition}\label{def:voigt_sol}
   Let $T>0$.  Suppose $\bu_0\in V$ and $\theta_0\in L^2$.  We say that $(\bu,\theta)$ is a \textit{weak solution} to the problem $P_{0,0}^\alpha$ on the interval $[0,T]$ if for all test functions $\Phi$, $\vphi$ chosen as in \eqref{test_fcns}, $(\bu,\theta)$ satisfies
\begin{subequations}\label{bouss_v_weak}
\begin{align}
   &\quad\notag
   -\int_0^T(\bu(s),\Phi'(s))\,ds-\alpha^2\int_0^T((\bu(s),\Phi'(s)))\,ds
   +\sum_{j=1}^2\int_0^T(u^j\bu,\partial_j\Phi)\,ds
   \\&\label{bouss_v_weak_mo}
    = (\bu_0,\Phi(0))+\alpha^2((\bu_0,\Phi(0)))+\int_0^T(\theta(s)\be_2,\Phi(s))\,ds,
   \\&\label{bouss_v_weak_den}
   -\int_0^T(\theta(s),\vphi'(s))\,ds + \int_0^T(\theta\bu,\nabla\vphi)\,ds = (\theta_0,\vphi(0)).
\end{align}
\end{subequations}
 and furthermore, $\bu\in C([0,T],V)$, $\pd{\bu}{t}\in L^\infty([0,T],V)$, $\theta\in L^\infty([0,T],L^2)$, $\theta\in C_w([0,T],L^2)$ and $\pd{\theta}{t}\in L^\infty([0,T],H^{-2}) $.
\end{definition}
\begin{remark}
Following similar arguments as those for the NSE presented in \cite{Temam_2001_Th_Num} one can show that this definition is equivalent to the functional equation
\begin{subequations}\label{e:functional_voigt}
  \begin{align}\label{e:funct_1_voigt}
  (I+\alpha^2A)\pd{\bu}{t} +  B(\bu,\bu) &= P_\sigma(\theta\be_2)\quad \mbox{in}\quad L^2([0,T], V')\quad
  \mbox{and}\\\label{e:funct_2_voigt}
   \pd{\theta}{t} + \mathcal{B}(\bu,\theta) &= 0\quad \mbox{in}\quad L^2([0,T], H^{-2}).
   \end{align}
   \end{subequations}

\end{remark}
\begin{theorem}\label{bouss_v_exist}
   Let $\bu_0\in V$, $\theta_0\in L^2$.  Then there exists a solution to $P^\alpha_{0,0}$, in the sense of Definition \ref{def:voigt_sol}.  Furthermore, if $\theta_0\in L^\infty$, then $\theta\in L^\infty([0,T], L^\infty)$.
  \end{theorem}

\begin{proof}  We use the notation laid out in Section \ref{sec:Pre}.
   Let us consider the Galerkin approximation to $P^\alpha_{0,0}$ (or equivalently, \eqref{e:functional_voigt}) given by
 \begin{subequations}\label{bouss_v_Gal}
\begin{alignat}{2}
\label{bouss_v_Gal_mo}
(I+\alpha^2A)\partial_t\bu_n + P_nB(\bu_n,\bu_n) &= P_nP_\sigma(\theta_n \be_2),\\
\label{bouss_v_Gal_den}
\partial_t\theta_n + P_n(\nabla\cdot(\bu_n \theta_n)) &=0,\\
\label{bouss_v_Gal_in}
\bu_n(0)=P_n\bu_0,\quad \theta_n(0)&=P_n\theta_0.
\end{alignat}
\end{subequations}
   This is a finite dimensional system of ODEs in $H_n$ with quadratic polynomial non-linearity, and therefore it has a unique local solution in $C^1([0,T_n),H_n)$ for some $T_n>0$.  Let $[0,T_n^*)$ be the maximal interval of existence and uniqueness of solutions to \eqref{bouss_v_Gal}.  We show below that $T_n^*=\infty$ for every $n$.

Taking the inner product of \eqref{bouss_v_den} with $\theta_n$, using Lemma \ref{B_theta:prop}, and integrating in time, we find that for $t\in[0,T_n^*)$,
\begin{align}\label{v_theta_L2}
  |\theta_n(t)| =|\theta_n(0)|\leq |\theta_0|.
\end{align}
\noindent Next, we take the inner product of \eqref{bouss_v_mo} with $\bu_n$ and use Lemma \ref{B:prop} to find
\begin{align}
   \frac{1}{2}\frac{d}{dt}(|\bu_n|^2+\alpha^2\|\bu_n\|^2)
   &=\notag
   (\theta_n\be_2,\bu_n)
   \leq
   |\theta_n||\bu_n|
     \\&\leq
   \label{bouss_L2_est}
  |\theta_0|\sqrt{|\bu_n|^2+\alpha^2\|\bu_n\|^2}
\end{align}
Consequently, we have for $t\in[0,T_n^*)$,
\begin{align}\label{v_u_en}
   |\bu_n(t)|^2+\alpha^2\|\bu_n(t)\|^2
   \leq
   |\bu_0|^2+\alpha^2\|\bu_0\|^2+t^2|\theta_0|^2.
\end{align}
According to \eqref{v_theta_L2} and \eqref{v_u_en},  we see that if $T_n^*<\infty$, then $\|\bu_n\|$ and $|\theta_n|$ are both bounded in time on $[0,T_n^*)$, and thus the solutions can be continued beyond $T_n^*$, contradicting the definition of $T_n^*$ as the maximal time of existence.  Thus, $T_n^*=\infty$ for all $n\in\nN$.  Next, we find bounds on the time derivatives.  From now on, we work on the interval $[0,T]$, where $T$ was arbitrarily given in the statement of the theorem.
Using Lemma \ref{B:prop}, along with \eqref{poincare} and \eqref{v_theta_L2}, we have
\begin{align}
   \norm{(I+\alpha^2A)\frac{d\bu_n}{dt}}_{V'}
   &\leq\label{v_Au_t+u_t}
   \sup_{\|\bw\|=1}\abs{\pair{B(\bu_n,\bu_n)}{P_n\bw}}+
   \sup_{\|\bw\|=1}\abs{\pair{\theta_n\be_2}{P_n\bw}}
   \\&\leq\notag
   \sup_{\|\bw\|=1}|\bu_n|\|\bu_n\|\|\bw\|+
   \sup_{\|\bw\|=1}|\theta_n||\bw|
   \\&\leq\notag
   |\bun|\|\bun\| + \lambda_1^{-1/2}|\thetan|
   \\&\leq\notag
   |\bu_n|\|\bu_n\|+\lambda^{-1/2}|\theta_0|.
\end{align}
Thanks to \eqref{v_u_en} and \eqref{v_Au_t+u_t} we obtain that $(I+\alpha^2A)\frac{d\bu_n}{dt}$ is uniformly bounded in $L^\infty([0,T],V')$, which implies that $\frac{d\bu_n}{dt}$ is uniformly bounded in $L^\infty([0,T],V)$, with respect to $n$.
Similarly, using Lemma \ref{B_theta:prop}, we obtain
\begin{equation}
\norm{\pd{\theta_n}{t}}_{H^{-2}} \leq |\theta_n||\bu_n|^{1/2}\|\bu_n\|^{1/2},
\end{equation}
which implies that $\pd{\theta_n}{t}$ is bounded in  $L^\infty([0,T],H^{-2})$ independently of $n$ by virtue of \eqref{v_theta_L2} and \eqref{v_u_en}.

The above bounds allow us to use the Banach-Alaoglu Theorem and the Aubin Compactness Theorem  (see, e.g., \cite{Temam_2001_Th_Num, Constantin_Foias_1988}) to extract a subsequence, which we still write as $(\bu_n,\theta_n)$, and elements $\bu$ and $\theta$, such that
   \begin{subequations}\label{wk_conv_Gal}
\begin{align}
\label{st_u_L2H_Gal}
\bu_n\maps\bu &\quad\text{strongly in }L^2([0,T],H),
\\\label{wk_u_L2V_Gal}
\bu_n\rightharpoonup\bu &\quad\text{weakly in }L^2([0,T],V)\text{ and weak-$*$ in }L^\infty([0,T],V),\\
\label{wk_du_L2V_Gal}
\pd{\bu_n}{t}\rightharpoonup\pd{\bu}{t} &\quad\text{ weak-$*$ in }L^\infty([0,T],V),
\\\label{wk_theta_LiH_Gal}
\theta_n\rightharpoonup\theta &\quad\text{weakly in }L^2([0,T],L^2)\text{ and weak-$*$ in }L^\infty([0,T],L^2),
\\\label{wk_dtheta_LiH_Gal}
\pd{\theta_n}{t}\rightharpoonup\pd{\theta}{t} &\quad\text{ weak-$*$ in }L^\infty([0,T],H^{-2}).
\end{align}
\end{subequations}

   Next, for arbitrary $\vphi$ and $\Phi$, chosen as in \eqref{test_fcns},  let us take the  inner product of $\eqref{bouss_v_Gal_mo}$ with $\Phi$, and of \eqref{bouss_v_Gal_den} with $\vphi$ and integrate in time on $[0,T]$.  After integrating by parts several times, we have
\begin{subequations}\label{bouss_v_Gal_weak}
\begin{align}
   &\quad
   -\int_0^T(\bu_n(s),\Phi'(s))\,ds-\alpha^2\int_0^T((\bu_n(s),\Phi'(s)))\,ds
   +\sum_{j=1}^2\int_0^T(u_n^j\bu_n,P_n\partial_j\Phi)\,ds
   \\&\notag
    = (\bu_n(0),\Phi(0))+\alpha^2((\bu_n(0),\Phi(0)))+\int_0^T(\theta_n(s)\be_2,\Phi(s))\,ds,
   \\&
   -\int_0^T(\theta_n(s),\vphi'(s))\,ds + \int_0^T(\theta_n\bu_n,P_n\nabla\vphi)\,ds = (\theta_0,\vphi(0)),
\end{align}
\end{subequations}
where we have again denoted $'\equiv \pd{}{s}$.  We would like to pass to the limit as $n\maps\infty$ to obtain \eqref{bouss_v_weak}.  The convergence of the linear terms is straight-forward, thanks to \eqref{wk_conv_Gal}.  As for the non-linear terms, notice that the convergence in \eqref{wk_conv_Gal} is  in stronger than the convergence in \eqref{wk_conv}, and so the convergence of the non-linear terms follows just as in the proof of Theorem \ref{exist_weak_visc} (note that $P_n\Phi= \Phi$ and $P_n\vphi=\vphi$ for sufficiently large $n$, due to our choice of test functions in \eqref{test_fcns}). Thus, $(\bu,\theta)$ satisfies \eqref{bouss_v_weak}.  In particular, choosing $\phi$ and $\Phi$ to have compact support in $(0,T)$, we see that the equations of \eqref{e:funct_1_voigt} and \eqref{e:funct_2_voigt} are satisfied in the sense of distributions in time with values in $V'$ and $H^{-2}$, respectively.  Acting with  \eqref{e:funct_1_voigt} on $\Phi$, with \eqref{e:funct_2_voigt} on $\vphi$, and integrating in time on $[t_0,t_1]$, we find
\begin{subequations}\label{bouss_v_weak_integrated}
\begin{align}
   &\quad\label{bouss_v_weak_integrated_mo}
   -\int_{t_0}^{t_1}(\bu(s),\Phi'(s))\,ds-\alpha^2\int_{t_0}^{t_1}((\bu(s),\Phi'(s)))\,ds
   +\sum_{j=1}^2\int_{t_0}^{t_1}(u^j\bu,\partial_j\Phi)\,ds
   \\&\notag
    = (\bu(t_0),\Phi(t_0))+\alpha^2((\bu(t_0),\Phi(t_0)))
    - (\bu(t_1),\Phi(t_1))-\alpha^2((\bu(t_1),\Phi(t_1)))
    \\&\quad\notag
    +\int_{t_0}^{t_1}(\theta(s)\be_2,\Phi(s))\,ds,
   \\&\label{bouss_v_weak_integrated_den}
   -\int_{t_0}^{t_1}(\theta(s),\vphi'(s))\,ds + \int_{t_0}^{t_1}(\theta\bu,\nabla\vphi)\,ds = (\theta(t_0),\vphi(t_0))- (\theta(t_1),\vphi(t_1)).
\end{align}
\end{subequations}
Temporarily restricting our set of test functions to those which are compactly supported in time on $[0,T]$ and considering the case $t_0=0$ and $t_1=T$, it is easy to see from a simple density argument that $(\bu,\theta)$ satisfies the equations of \eqref{e:functional_voigt} in the sense of distributions, thanks to \eqref{bouss_v_weak_integrated}.  Next, allowing $\Phi(0)$, and $\vphi(0)$ to be arbitrary, but fixing $\Phi(T)=0$ and $\vphi(T)=0$, we act on $\Phi$ with \eqref{e:funct_1_voigt} and on $\vphi$ with \eqref{e:funct_2_voigt} and integrate on $[0,T]$, resulting equations in \eqref{bouss_v_weak}.  Comparing \eqref{bouss_v_weak} with \eqref{bouss_v_weak_integrated}, we find that that $\theta(0)=\theta_0$ in the sense of $H^{-1}$ and that $(I+\alpha^2A)\bu(0)=(I+\alpha^2A)\bu_0$ in the sense of $V'$.  Inverting $I+\alpha^2A$ gives $\bu(0)=\bu_0$.  Furthermore, we may send $t_1\maps t_0$ in \eqref{bouss_v_weak_integrated_den}, and use the density of $C^{\infty}(\nT^2)$ in $L^2(\nT^2)$, as well as the boundedness of $\theta$ in $L^\infty([0,T],L^2(\nT^2))$, to show that $\theta\in C_w([0,T],L^2(\nT^2))$.  Next, since we have $\bu,\frac{d}{dt}\bu\in L^\infty([0,T],V)\hookrightarrow L^2([0,T],V)$, it follows that $\bu\in C([0,T],V)$ by the Sobolev Embedding Theorem.  Thus, we have shown a weak solution exists in the sense of Definition \ref{def:voigt_sol}.  Finally, one may show that if $\theta_0\in L^\infty(\nT^2)$, then $\theta\in L^\infty([0,T],L^\infty)$, by following Step 4 of the proof of Theorem \ref{exist_weak_visc} line-by-line.
\end{proof}

\begin{theorem}[Uniqueness for the 2D Voigt model]\label{voigt_unique}  Let $T>0$ be arbitrary.  Suppose $\bu_0\in \cD(A)$ and $\theta_0\in L^2(\nT^2)$.  Then there exists a unique solution to \eqref{bouss_v} in the sense of Definition \ref{def:voigt_sol}.  Furthermore, it holds that $\bu\in L^\infty([0,T],\cD(A))$.
\end{theorem}
\begin{proof}
Here, we only sketch the proof, since the ideas a similar to those given above.  The existence of solutions to \eqref{bouss_v} has already been established in Theorem \ref{bouss_v_exist}.  Thanks to the hypothesis $\bu_0\in \cD(A)$, it is straight-forward to show that $\bu\in C([0,T],\cD(A))$ using, e.g., the methods of  \cite{Larios_Titi_2009} and similarly that $\frac{d\theta}{dt} \in L^2([0,T], H^{-1})$.   One can the prove the uniqueness of solutions  by following the proof of Theorem \ref{uniqueness_visc} almost line by line.  Only some slight modifications to the handling of the terms involving $\|\bu\|$, and in using the parameter $\alpha^2$ rather than $\nu$ is needed.
\end{proof}

\begin{theorem}[Convergence as $\alpha\maps0$]\label{t:Convergence} Given initial data $(\bu_0, \theta_0)\in (H^3(\nT^2)\cap V)\times H^3(\nT^2)$, and $(\bu_0^\alpha,\theta_0^\alpha)\in (H^3(\nT^2)\cap V)\times H^3(\nT^2)$, let $(\bu,\theta)$ and $(\bu^\alpha,\theta^\alpha)$ be the corresponding solutions to the problems $P_{0,0}^0$ and $P_{0,0}^\alpha$, respectively.  Choose an arbitrary  $T\in (0,T_{\text{max}})$, where $T_{\text{max}}$ is the maximal time for which a solution to  the problem  $P_{0,0}^0$ exists and is unique.  Suppose that $\bu_0^\alpha\maps\bu_0$ in $V$ and  $\theta_0^\alpha\maps\theta_0$ in $L^2(\nT^2)$.  Then $\bu^\alpha\maps\bu$ in $L^2([0,T],V)$ and  $\theta^\alpha\maps\theta$ in $L^2([0,T],L^2(\nT^2))$.
\end{theorem}
\begin{proof}
Here, for simplicity, we only work formally, but note that the results can be made rigorous by using the techniques discussed above.   Under the hypotheses on the initial conditions, it was proven in \cite{Chae_Nam_1997} that there exists a time $T>0$ and a unique $(\bu,\theta) \in C([0,T],H^3({\nT^2})\cap V)\times C([0,T],H^3({\nT^2})\cap V)$ solving the problem $P_{0,0}^0$, (in particular, it holds that $T_{\text{max}}>0$).  Thanks to Theorem \eqref{bouss_v_exist}, we know that there also exists a unique solution to the problem $P_{0,0}^\alpha$, namely  $(\bu^\alpha,\theta^\alpha)\in C([0,T], V)\times C([0,T],L^2({\nT^2}))$.  Subtracting the corresponding equations from these problems (that is, \eqref{bouss} with $\nu=\kappa=0$ and \eqref{bouss_v}) yields
   \begin{subequations}\label{bouss_subtract}
\begin{align}
\label{bouss_subtract_mo}
-\alpha^2 \frac{d}{dt}\triangle\bu^\alpha+\frac{d}{dt}(\bu-\bu^\alpha) &=-B (\bu-\bu^\alpha,\bu)
-B (\bu^\alpha, \bu-\bu^\alpha) +P_\sigma((\theta^\alpha-\theta) \be_2), \\
\label{bouss_subtract_den}
\frac{d}{dt}(\theta^\alpha-\theta)
&=- ((\bu-\bu^\alpha)\cdot\nabla)\theta -(\bu^\alpha\cdot\nabla)(\theta-\theta^\alpha).
\end{align}
\end{subequations}
Let us take the inner product of \eqref{bouss_subtract_mo} with $\bu^\alpha-\bu$ and  of \eqref{bouss_subtract_den} with $\theta^\alpha-\theta$, and add the results.  After integrating by parts and rearranging the terms, we find
\begin{align}
&\quad\label{conv_diff_est}
 \frac{1}{2}\frac{d}{dt}\pnt{\alpha^2\|\bu-\bu^\alpha\|^2+|\bu-\bu^\alpha|^2
 +|\theta-\theta^\alpha|^2}
 \\&=\notag
-(B (\bu-\bu^\alpha,\bu),\bu^\alpha-\bu)
+((\theta^\alpha-\theta) \be_2,\bu^\alpha-\bu)
\\&\quad\notag
-(((\bu-\bu^\alpha)\cdot\nabla)\theta,\theta^\alpha-\theta)
-\alpha^2\pair{\triangle\bu_t}{\bu-\bu^\alpha}
\\&\leq \notag
\|\nabla\bu\|_{L^\infty}|\bu-\bu^\alpha|^2
+|\theta^\alpha-\theta||\bu^\alpha-\bu|
\\&\quad\notag
\|\nabla\theta\|_{L^\infty}|\bu-\bu^\alpha||\theta^\alpha-\theta|
-\alpha^2\pair{\triangle\bu_t}{\bu-\bu^\alpha}
\\&\leq \notag
K(|\bu-\bu^\alpha|^2
+|\theta^\alpha-\theta|^2)
-\alpha^2\pair{\triangle\bu_t}{\bu-\bu^\alpha},
\end{align}
where we have used Young's inequality and the fact that $\|\nabla\bu\|_{L^\infty},\|\nabla\theta\|_{L^\infty}<\infty$.  It remains to estimate the integral on the left-hand side of the equality. Using the fact that $(\bu,\theta)$ satisfies \eqref{bouss_den}, we have
\begin{align}
   &\quad \label{conv_est_H3}
   -\alpha^2\pair{ \triangle\bu_t}{\bu-\bu^\alpha}
\\&=\notag
  -\alpha^2\pair{ \triangle[-\bu\cdot\nabla \bu -\nabla p + \theta \be_2]}{\bu-\bu^\alpha}
  \\&=\notag
\alpha^2\pair{ \triangle\bu\cdot\nabla \bu+2(\nabla\bu\cdot\nabla) \nabla\bu+\bu\cdot\nabla \triangle\bu - \triangle\theta \be_2}{\bu-\bu^\alpha}
 \\&\leq\notag
 C\alpha^2|\triangle \bu|\|\nabla\bu\|_{L^\infty}|\bu-\bu^\alpha|+\|\bu\|_{L^\infty}|\nabla\triangle \bu|
 |\bu-\bu^\alpha| + |\triangle\theta||\bu-\bu^\alpha|
  \\&\leq\notag
 C\alpha^2\|\bu\|_{H^2}\|\bu\|_{H^3}|\bu-\bu^\alpha|+C\|\bu\|_{H^2}\| \bu\|_{H^3}
|\bu-\bu^\alpha| + C\|\theta\|_{H^2}|\bu-\bu^\alpha|
  \\&\leq\notag
 \alpha^2K|\bu-\bu^\alpha|.
\end{align}
For the second equality, we used \eqref{deRham}.
Combining \eqref{conv_diff_est} with \eqref{conv_est_H3} and using Gr\"onwall's inequality yields
\begin{align}
  &\quad\notag
  \alpha^2\|\bu(t)-\bu^\alpha(t)\|^2+|\bu(t)-\bu^\alpha(t)|^2+|\theta(t)-\theta^\alpha(t)|^2
 \\&\leq
 C\pnt{\alpha^2\|\bu_0-\bu^\alpha_0\|^2+|\bu_0-\bu^\alpha_0|^2
 +|\theta_0-\theta^\alpha_0|^2}e^{K(\alpha^2+\alpha)t}.
\end{align}
Thus, if $\bu^\alpha_0\maps \bu_0$ in $V$ and $\theta_0^\alpha\maps \theta_0$ in $L^2(\nT^2)$, as $\alpha\maps 0$, (in particular, if $\bu^\alpha_0= \bu_0$ and $\theta_0^\alpha= \theta_0$ for all $\alpha>0$), then $\bu^\alpha\maps \bu$ in $L^\infty([0,T],V)$ and $\theta^\alpha\maps \theta$ in $L^\infty([0,T],L^2)$, as $\alpha\maps 0$.
\end{proof}

\begin{theorem}[Blow-up criterion]\label{t:blow_up_criterion}
With the same notation and assumptions of Theorem \ref{t:Convergence}, suppose that for some $T_*<\infty$, we have
\begin{align}\label{blow_up_ineq}
   \sup_{t\in[0,T_*)}\limsup_{\alpha\maps0}\alpha^2\|\bu^\alpha(t)\|^2>0.
\end{align}
Then the solutions to $P_{0,0}^0$ become singular in the time interval $[0,T_*)$.
\end{theorem}
\begin{proof}
   To get a contradiction, suppose that $(\bu,\theta)$ stays bounded in $H^3\cap V({\nT^2})\times H^3({\nT^2})$ but that \eqref{blow_up_ineq}.   Taking the inner product of the momentum equation with $\bu^\alpha$ and integrating, we find
   \begin{align*}
      \alpha^2\|\bu^\alpha(t)\|^2+|\bu^\alpha(t)|^2
      &=\alpha^2\|\bu_0\|^2+|\bu_0|^2+2\int_0^t(\theta^\alpha(s)\be_2,\bu^\alpha(s))\,ds.
   \end{align*}
Taking the $\limsup$ as $\alpha\maps 0^+$ then by virtue of Theorem \ref{bouss_v_exist} and Theorem \ref{t:Convergence} we have
\begin{align}\label{blow_up_limsup}
      \limsup_{\alpha\maps0^+}\alpha^2\|\bu^\alpha(t)\|^2+|\bu(t)|^2
      &=|\bu_0|^2+2\int_0^t(\theta(s)\be_2,\bu(s))\,ds.
   \end{align}

However, given the hypotheses on the initial data, and the well-posedness results of \cite{Chae_Nam_1997}, it is straight-forward to prove the following energy equality:
\begin{align*}
      |\bu(t)|^2&=|\bu_0|^2+2\int_0^t(\theta(s),\bu(s))\,ds,
   \end{align*}
   so that \eqref{blow_up_limsup} contradicts \eqref{blow_up_ineq}.
\end{proof}

\section{The 3D Boussinesq-Voigt Equations} \label{s:3D_Bouss_Voigt}

We now briefly outline an extension of the previous results to the case of the three dimensional Boussinesq-Voigt equations.  The details are very similar to the 2D case, so we only prove formal \textit{a priori} estimates.  In order to control the higher-order derivatives, we add a diffusion term to the transport equation.    This approach is similar to that used in \cite{Larios_Titi_2009,Larios_Titi_2010_MHD,Catania_Secchi_2009,Catania_2009} to prove global well-posedness for two Voigt-regularizations of the 3D MHD equations.  We consider the following system, written in functional form, which we refer to as $P_{0,\kappa}^\alpha$.
\begin{subequations}\label{bouss_v_3D}
  \begin{align}
  \label{bouss_v_3D_mom}
  (I+\alpha^2A)\pd{\bu}{t} +  B(\bu,\bu) &= P_\sigma(\theta\be_3),
  \\\label{bouss_v_3D_den}
   \pd{\theta}{t} +\mathcal{B}(\bu,\theta) &=  \kappa \triangle\theta,
   \\\label{bouss_v_3D_IC}
\bu(0)=\bu_0,\quad\theta(0)&=\theta_0.
   \end{align}
   \end{subequations}
   \begin{remark}
      Note that one could also control the higher-order derivatives by adding the Voigt-term $-\beta^2\triangle\frac{d}{dt}\theta$, $\beta>0$ to the left-hand side of \eqref{bouss_v_3D_den} to allow for case when $\kappa=0$.  Although the resulting regularized system is well-posed,
      for the sake of brevity, we do not pursue this type of additional Voigt-regularization here.  However, a similar idea has been investigated in the context of the MHD equations in \cite{Larios_Titi_2009} (cf. \cite{Larios_Titi_2010_MHD}), and also in \cite{Catania_Secchi_2009}.
   \end{remark}

\begin{definition}\label{def:bouss_v_3D}
   Let $\bu_0\in V\cap H^3(\nT^3)$, $\theta_0\in L^2(\nT^3)$.  For a given $T>0$, we say that $(\bu,\theta)$ is a \textit{solution}
   to the problem $P_{0,\kappa}^\alpha$ (in three-dimensions) on the interval $[0,T]$ if it satisfies \eqref{bouss_v_3D_mom} in the sense of $L^2([0,T],V)$ and \eqref{bouss_v_3D_den} in the sense of $L^2([0,T],H^{-1})$.
Furthermore, $\bu\in C([0,T],V\cap H^3)$, $\pd{\bu}{t}\in L^\infty([0,T],\mathcal{D}(A))\cap L^2([0,T],V\cap H^3)$, $\theta\in L^2([0,T],H^1)\cap C_w([0,T],L^2)$ and $\pd{\theta}{t}\in L^2([0,T],H^{-1})$.
\end{definition}

 \begin{theorem}
    Let $\bu_0\in V\cap H^3(\nT^3)$, $\theta_0\in L^2(\nT^3)$, and let $T>0$ be arbitrary.  Then there exists a solution to \eqref{bouss_v_3D} in the sense of Definition \ref{def:bouss_v_3D}.   Furthermore, if $\theta_0\in L^p(\nT^3)$ for some $p\in[2,\infty]$, then $\theta\in L^\infty([0,T],L^p)$.  In the case $\theta_0\in L^\infty(\nT^3)$, the solution is unique.
 \end{theorem}
 \begin{proof}
    As mentioned above, we only establish formal \textit{a priori} estimates here.  Suppose for a moment  that $\theta_0\in L^p(\nT^3)$.  Formally taking the inner product of \eqref{bouss_v_3D_den} with $|\theta|^{p-2}\theta$, $p\in[2,\infty)$, we find as above that
    \begin{align}\label{bouss_v_3D_theta_Lp}
\frac{1}{p}\frac{d}{dt}\|\theta\|_{L^p}^p+\kappa(p-1)\int_{\nT^3}|\nabla\theta|^2|\theta|^{p-2}\,d\bx
=0.
\end{align}
Dropping the term involving $\kappa$, integrating in time, and sending $p\maps\infty$, we find for all $p\in[2,\infty]$,
\begin{align*}
\|\theta(t)\|_{L^p}
\leq
\|\theta_0\|_{L^p}.
\end{align*}
On the other hand, setting $p=2$ in \eqref{bouss_v_3D_theta_Lp} and integrating in time, we find
\begin{align*}
|\theta(t)|^2+2\kappa\int_0^t\|\theta(s)\|^2\,ds
\leq
|\theta_0|^2.
\end{align*}
Next, following similar steps as in the derivation of  \eqref{bouss_L2_est} and \eqref{v_u_en}, we find for $t\in[0,T]$,
    \begin{align*}
       |\bu(t)|^2+\alpha^2\|\bu(t)\|^2 \leq |\bu_0|^2+\alpha^2\|\bu_0\|^2 + T^2|\theta_0|^2 :=(K_{\alpha,1})^2.
    \end{align*}
    By formally taking the inner product of \eqref{e:funct_1_voigt} with $A\bu$, we find
\begin{align*}
       \frac{1}{2}\frac{d}{dt}(\|\bu\|^2+\alpha^2|A\bu|^2)
       &= -(B(\bu,\bu),A\bu)+(\theta \be_3,A\bu)
       \\&\leq C\|\bu\||A\bu|^2+|\theta| |A\bu|
       \leq C K_{\alpha,1}\alpha^{-1}|A\bu|^2+|\theta_0| |A\bu|
       \\&\leq C(1+K_{\alpha,1}\alpha^{-1}) |A\bu|^2+|\theta_0|^2.
    \end{align*}
Using Gr\"onwall's inequality, we obtain a constant $K_{\alpha,2}=(\|\bu_0\|^2+\alpha^2|A\bu_0|^2+\frac{2K_{\alpha,1}}{T})e^{c_0T}$, where $c_0:=C\alpha^{-2}+K_{\alpha,1}\alpha^{-3}$, such that $\|\bu(t)\|^2+\alpha^2|A\bu(t)|^2\leq K_{\alpha,2}$ for a.e. $t\in[0,T]$.  Next, we formally take the inner product of \eqref{e:funct_1_voigt} with $A^2\bu$ (recalling that, in the periodic case, $-\triangle = A$), to find
\begin{align}
&\quad\notag
       \frac{1}{2}\frac{d}{dt}(|A\bu|^2+\alpha^2\|A\bu\|^2)
=
       -\sum_{j=1}^3(u^j\partial_j\bu,\triangle^2\bu)+(\theta \be_3,\triangle^2\bu)
\\&=\notag
       -\sum_{j=1}^3(\triangle u^j\cdot\partial_j\bu,\triangle \bu)
       -2 \sum_{i,j=1}^3(\partial_i u^j\cdot\partial_i\partial_j\bu,\triangle \bu)
       -\sum_{j=1}^3(u^j\partial_j\triangle \bu,\triangle \bu)
       \\&\quad\notag
       -((\theta \be_3,\triangle\bu))
\\&\leq\label{bouss_v_3D_u_H3}
       C|A\bu|^2\|A\bu\|+\|\theta\| \|A\bu\|
\leq
       C\pnt{(K_{\alpha,2})^4+\|A\bu\|^2+\|\theta\|^2}.
    \end{align}

These \textit{a priori} estimates can be used to form a rigorous argument as follows.  In the case $p=2$, existence can be proven by using, e.g., the Galerkin method as in the proof of theorem \eqref{bouss_v_exist}, substituting the \textit{a priori} estimates established in this section as necessary.   Passage to the limit, estimates on the time derivatives, and the continuity properties of Definition \ref{def:bouss_v_3D} can be established using similar ideas to those used in the proofs of some of the previous theorems.  For the case $\theta_0\in L^p(\nT^3)$, with $p\in(2,\infty)$, we begin by smoothing $\theta_0$ (e.g., by convolving it with a mollifier) to get smooth functions $\theta_0^\epsilon$ for each $\epsilon>0$, which converge to $\theta_0$ as $\epsilon\maps0$ in several relevant norms.  Clearly $\theta_0^\epsilon\in L^2(\nT^3)$.  Thus, thanks to the existence of solutions for the $p=2$ case, there exists a solution $(\bu^\epsilon,\theta^\epsilon)$ to \eqref{bouss_v_3D} with initial data $(\bu_0,\theta_0^\epsilon)$ such that $\theta^\epsilon \in L^\infty([0,T],L^2)$ and $\bu^\epsilon\in C([0,T],V\cap H^3)$.  One may then show that $\theta^\epsilon \in L^\infty([0,T],H^2)$, e.g., by using the Galerkin method and deriving straight-forward \textit{a priori} estimates on higher derivatives.  Since in three dimensions $L^\infty([0,T],H^2)$ is an algebra, it follows that $|\theta^\epsilon|^{p-2}\theta^\epsilon\in L^\infty([0,T],H^2)$, so that the above \textit{a priori} estimates can be established rigorously for $(\bu^\epsilon,\theta^\epsilon)$.  Furthermore, the bounds can be made to be independent of $\epsilon$.  Standard arguments show that one can extract subsequences of $(\bu^\epsilon,\theta^\epsilon)$  which converge in several relevant norms as $\epsilon\maps0$ to a solution $(\bu,\theta)$ of \eqref{bouss_v_3D} corresponding to initial data $(\bu_0,\theta_0)$.  Taking the limit as $\epsilon\maps0$, one may show that  $\theta\in L^\infty([0,T],L^p)$. Finally, in the case $p=\infty$, one may employ, e.g., the  Hopf-Stampacchia technique used in Step 4 of the proof of Theorem \ref{exist_weak_visc}.

With the above \textit{a priori} estimates formally established, we now show the uniqueness of solutions to \eqref{bouss_v_3D} under the additional hypothesis that $\theta_0\in L^\infty(\nT^3)$.  Letting $(\bu_1,\theta_1)$ and $(\bu_2,\theta_2)$ be two solutions to \eqref{bouss_v_3D} with initial data $(\bu_0,\theta_0)$.  Let us write $\bud:=\bu_1-\bu_2$ and $\thetad=\theta_1-\theta_2$.  As in the proof of Theorem \ref{uniqueness_visc} we also write $\xid=\triangle^{-1} \theta$ and $\xi_i=\triangle^{-1}\theta_i$, $i=1,2$, subject to the side condition $\int_{\nT^3} \xi_i\,d\bx=0$ for $i=1,2$.  Following nearly identical steps to the derivation of \eqref{bouss_u_diff}, we find,
\begin{align}
      \frac{1}{2}\frac{d}{dt}(|\bud|^2  + \alpha^2\|\bud\|^2)
      &\leq\label{bouss_v_u_diff}
      C_{\alpha,\kappa} K_{\alpha,1}(|\bud|^2+\alpha^2\|\bud\|^2+\|\xid\|^2).
   \end{align}
   Similarly, following nearly identical steps to the derivation of \eqref{bouss_diff_en_xi}, we find,
   \begin{align}
\frac{1}{2}\frac{d}{dt}\|\xid\|^2   +\kappa|\triangle\xid|^2
      &=\notag
      \pair{\bud\triangle\xi_1}{\nabla\xid}
      +\sum_{j=1}^3\pair{\partial_j\bu_2}{\nabla\xid\partial_j\xid}
      \\&\leq
      |\bud|\|\theta_1\|_{L^\infty}\|\xid\|
      +\|\bu_2\|_{H^3}\|\xid\|^2
      \leq\label{bouss_v_xi_diff}
      K_{\alpha,4}(|\bud|^2+\|\xid\|^2),
\end{align}
where $K_{\alpha,4} := \max\set{2\|\theta_0\|_{L^\infty},K_{\alpha,3}}<\infty$.  Uniqueness now follows by adding \eqref{bouss_v_u_diff} and \eqref{bouss_v_xi_diff} and using Gr\"onwall's inequality.
  \end{proof}
\section{Appendix} \label{s:app}
We prove the inequality \eqref{brezis}.  The proof is based on the proof of the Brezis-Gallouet inequality \cite{Brezis_Gallouet_1980} and follows almost line-by-line the proof given in \cite{Cao_Titi_2009}.  For $\bw\in\mathcal{D}(A)$, let us write
\begin{align*}
   \bw = \sum_{\bk\in\nZ^2\setminus(0,0)}a_\bk\bw_\bk
\end{align*}
where $\bw_\bk$ are the (normalized) eigenfunctions of $A$ (see Section \ref{sec:Pre}) and $a_\bk:=(\bw,\bw_\bk)$.   Choose $M=(e^{1/\epsilon^{1/4}}-1)^{1/2}$ for a given $\epsilon>0$, sufficiently small so that $M>1$.  We have
{\allowdisplaybreaks
\begin{align*}
\|\bw\|_{L^\infty}
&\leq
\sum_{\bk\in\nZ^2\setminus(0,0)}|a_\bk|
=
\sum_{0<|\bk|\leq M}|a_\bk|
+\sum_{|\bk|> M}|a_\bk|
\\&=
\sum_{0<|\bk|\leq M}\frac{(1+|\bk|^2)^{1/2}}{(1+|\bk|^2)^{1/2}}|a_\bk|
+\sum_{|\bk|> M}\frac{(1+|\bk|^2)}{(1+|\bk|^2)}|a_\bk|
\\&\leq
\pnt{\sum_{0<|\bk|\leq M}(1+|\bk|^2)|a_\bk|^2}^{1/2}
\pnt{\sum_{0<|\bk|\leq M}\frac{1}{(1+|\bk|^2)}}^{1/2}
\\&\quad
+\pnt{\sum_{|\bk|> M}(1+|\bk|^2)^2|a_\bk|^2}^{1/2}
\pnt{\sum_{|\bk|> M}\frac{1}{(1+|\bk|^2)^2}}^{1/2}
\\&\leq
C\|\bw\|
\pnt{\int_{|\bx|\leq M}\frac{d\bx}{(1+\bx^2)}}^{1/2}
+C|A\bw|
\pnt{\int_{|\bx|> M}\frac{d\bx}{(1+|\bx|^2)^2}}^{1/2}
\\&=
C\|\bw\|
\pi\log(1+M^2)
+C|A\bw|\frac{\pi}{1+M^2}
\\&= C\pnt{\|\bw\|\epsilon^{-1/4}+ |A\bw|e^{-1/\epsilon^{1/4}}}.
\end{align*}
} 

\section*{Acknowledgements}
The authors are thankful for the warm hospitality
of  the Institute for Mathematics and its
Applications (IMA), University of Minnesota, where
part of this work was completed. This work was
supported in part by the NSF grants
no.~DMS-0708832, DMS-1009950. E.S.T. also
acknowledges the kind hospitality of the Freie
Universit\"at - Berlin, and the support of the
Alexander von Humboldt Stiftung/Foundation and the
Minerva Stiftung/Foundation.
\begin{scriptsize}
\bibliographystyle{amsplain}
\bibliography{LariosBiblio}
\end{scriptsize}
\end{document}